\documentclass{amsart}
\usepackage{hyperref}
\usepackage[left=1.9cm,right=1.9cm,top=2cm,bottom=2.5cm]{geometry}
\usepackage{mathabx}
\usepackage{tikz}
\usepackage{mathrsfs}
\usepackage{cite}
\usetikzlibrary{arrows, cd}
\tikzset{> = stealth}
\tikzcdset{arrow style = tikz}

\newtheorem{thm}{Theorem}[section]
\newtheorem{prp}[thm]{Proposition}
\newtheorem{conjecture}[thm]{Conjecture}
\newtheorem{lem}[thm]{Lemma}
\newtheorem{cor}[thm]{Corollary}
\theoremstyle{change}
\theoremstyle{nonumberplain}

\theoremstyle{definition}
\newtheorem{dfn}[thm]{Definition}

\theoremstyle{remark}

\newtheorem{remark}[thm]{Remark}

\numberwithin{equation}{section}

\begin{document}

\title{Koszul self duality of manifolds}
\author{Connor Malin}
 \address{University of Notre Dame}
 \email{cmalin@nd.edu}
\maketitle

%%
%% Now edit the following to give your name and address:
%% 

\begin{abstract}
    We show that Koszul duality for operads in $(\mathrm{Top},\times)$ can be expressed via generalized Thom complexes. As an application, we prove the Koszul self duality of the right module $E_M$ associated to a framed manifold $M$. We discuss implications for factorization homology, embedding calculus, and confirm an old conjecture of Ching on the relation of Goodwillie calculus to manifold calculus.
\end{abstract}

\tableofcontents

%I THINK WE NEED TO CHANGE DEFINITION OF EM+, not functorial with respect to open embeddings

\section{Introduction}

Consider the topological category $\mathscr{M}\mathrm{fld}_n$ of $n$-dimensional smooth manifolds and smooth embeddings. This category has an important full subcategory $\mathscr{D}\mathrm{isk}_n$ for which the objects are finite disjoint unions of $\mathbb{R}^n$. Associated to a manifold $M \in \mathscr{M}\mathrm{fld}_n$ is the presheaf on $\mathscr{D}\mathrm{isk}_n$
\[\bigsqcup_{i \in I} \mathbb{R}^n \mapsto \mathrm{Emb}(\bigsqcup_{i \in I} \mathbb{R}^n, M ).\]

It was observed by Boavida de Brito--Weiss and Turchin \cite{brito_weiss_2013,Turchin_2013} that these presheaves determine the behavior of manifold calculus towers \cite{weiss_1999} associated to any enriched presheaf on $\mathscr{M}\mathrm{fld}_n$. In particular, it is known that in codimension $\geq 3$, the derived mapping space $\mathrm{Map}^h_{\mathscr{D}\mathrm{isk}_n}(\mathrm{Emb}(-, M ),\mathrm{Emb}(-, N ))$ has the homotopy type of $\mathrm{Emb}(M,N)$ \cite{goodwillie_klein_weiss_2001}.

These presheaves also occur in factorization homology -- a theory of integrating various types of $E_n$-algebras over $n$-manifolds. In particular, one observes that factorization homology only depends on the homotopy type of this presheaf. One calculational tool that appears in factorization homology is Poincaré-Koszul duality \cite{PKayala_francis_2019}. This intertwines the self duality of manifolds with the geometric bar-cobar duality of $E_n$-algebras introduced by Ayala-Francis. Their construction, which we denote \textbf{bar}, sends $E_n$-algebras to $E_n$-coalgebras.
%emphasize difference
Bar-cobar duality for $E_n$-algebras bears a strong resemblance to the classic form of Koszul duality, which we denote by $B$. However, there are important differences: if one applies \textbf{bar} to an $E_n$-algebra it produces an $E_n$-coalgebra, but if one applies $B$ to an $E_n$-algebra, it produces a $B(E_n)$-coalgebra. This distinction is unsurprising, since Koszul duality applies to algebras over any operad, while the bar-cobar duality of $E_n$-algebras that appears in factorization homology is unique to $E_n$-algebras. It is interesting to ask if there is an interpretation of Poincaré-Koszul duality which uses the more classical setting of Koszul duality. This is strongly suggested by the existence of ``Poincaré-Koszul'' equivalences over arbitrary operads, constructed separately by Ching, in the case of left modules, \cite[Proposition 6.1]{ching_2012} and Amabel in the case of algebras \cite[Main theorem]{amabel_2022}.

The essence of Koszul duality is exhibited in the commutativity of the diagram of algebraic operads:

 % https://q.uiver.app/?q=WzAsNixbMCwwLCJzX25cXG1hdGhybXtsaWV9Il0sWzEsMCwiXFxtYXRocm17cG9pc31fbiJdLFsyLDAsIlxcbWF0aHJte2NvbX0iXSxbMiwxLCJLKFxcbWF0aHJte2xpZX0pIl0sWzEsMSwic19uIEsoXFxtYXRocm17cG9pc31fbikiXSxbMCwxLCJzX24gSyhcXG1hdGhybXtjb219KSJdLFswLDUsIlxcc2ltZXEiLDAseyJzdHlsZSI6eyJoZWFkIjp7Im5hbWUiOiJub25lIn19fV0sWzEsNCwiXFxzaW1lcSIsMCx7InN0eWxlIjp7ImhlYWQiOnsibmFtZSI6Im5vbmUifX19XSxbNSw0XSxbNCwzXSxbMCwxXSxbMSwyXSxbMiwzLCJcXHNpbWVxIiwwLHsic3R5bGUiOnsiaGVhZCI6eyJuYW1lIjoibm9uZSJ9fX1dXQ==
\[\begin{tikzcd}
	{s_n\mathrm{lie}} & {\mathrm{pois}_n} & {\mathrm{com}} \\
	{s_n K(\mathrm{com})} & {s_n K(\mathrm{pois}_n)} & {K(\mathrm{lie})}
	\arrow["\simeq", no head, from=1-1, to=2-1]
	\arrow["\simeq", no head, from=1-2, to=2-2]
	\arrow[from=2-1, to=2-2]
	\arrow[from=2-2, to=2-3]
	\arrow[from=1-1, to=1-2]
	\arrow[from=1-2, to=1-3]
	\arrow["\simeq", no head, from=1-3, to=2-3]
\end{tikzcd}\]

 This diagram says the Koszul dual $K(-):=B(-)^\vee$ of the commutative operad is the Lie operad, the Koszul dual of the Lie operad is the commutative operad, and the Koszul dual of the n-Poisson operad is itself, up to a shift \cite[Chapter 13]{loday_vallette_2012}. This diagram encodes the following observations: the Koszul dual of a Lie algebra is a commutative algebra, the Koszul dual of a commutative algebra is a Lie algebra, and the Koszul dual of an n-Poisson algebra is a shifted n-Poisson algebra.\footnote{We use the convention that Lie algebras have a bracket of degree $-1$.} If we recall that $H_*(E_n)\cong \mathrm{pois}_n$ \cite{cohen_lada_may_1976}, we see the bar-cobar duality of $E_n$-algebras is reflected as a theorem about Koszul self duality of the homology of $E_n$.

 Koszul duality was originally formulated for operads in $(\mathrm{DGVect_\mathbb{Q}},\otimes)$, but was generalized to operads in $(\mathrm{Top}_*,\wedge)$ and $(\mathrm{Sp},\wedge)$ by Ching and Salvatore, independently \cite{ching_2005,salvatore_1998}. Koszul duality for topological operads has had great success, particularly in the realm of Goodwillie calculus. Arone-Ching used Koszul duality to endow the Goodwillie derivatives $\partial_* F$ of a functor $F:\mathrm{Top} \rightarrow \mathrm{Top}$ with the structure of a $\mathrm{lie}:=K(\mathrm{com})$ bimodule \cite{arone_ching_2011}. The point-set level construction of Koszul duality, combined with its interaction with Goodwillie calculus, often allows for explicit computation. One can deduce unstable results from stable techniques, not just in theory, but also in practice. The explicit cooperad structure on $B(\mathrm{com})$ was utilized by Behrens to calculate the ring of $\mathbb{F}_2$ Dyer-Lashof operations of the spectral Lie operad, which he used to give a novel calculation of the unstable homotopy groups of spheres up to the 19-stem \cite{behrens_2012}.\footnote{It is worth noting that recently Konovalov completed the calculation at odd primes without relying on the point-set construction of Koszul duality, in contrast with Behrens \cite{konovalov2023algebraic}.}

It is natural to ask if there is a diagram of operads in spectra which lifts the previous diagram of algebraic operads

% https://q.uiver.app/?q=WzAsNixbMCwwLCJzX25cXG1hdGhybXtsaWV9Il0sWzEsMCwiXFxTaWdtYV5cXGluZnR5XytFX24iXSxbMiwwLCJcXG1hdGhybXtjb219Il0sWzIsMSwiSyhcXG1hdGhybXtsaWV9KSJdLFsxLDEsInNfbiBLKFxcU2lnbWFeXFxpbmZ0eV8rRV9uKSJdLFswLDEsInNfbiBLKFxcbWF0aHJte2NvbX0pIl0sWzAsNSwiXFxzaW1lcSIsMCx7InN0eWxlIjp7ImhlYWQiOnsibmFtZSI6Im5vbmUifX19XSxbMSw0LCJcXHNpbWVxIiwwLHsic3R5bGUiOnsiaGVhZCI6eyJuYW1lIjoibm9uZSJ9fX1dLFs1LDRdLFs0LDNdLFswLDFdLFsxLDJdLFsyLDMsIlxcc2ltZXEiLDAseyJzdHlsZSI6eyJoZWFkIjp7Im5hbWUiOiJub25lIn19fV1d
\[\begin{tikzcd}
	{s_n\mathrm{lie}} & {\Sigma^\infty_+E_n} & {\mathrm{com}} \\
	{s_n K(\mathrm{com})} & {s_n K(\Sigma^\infty_+E_n)} & {K(\mathrm{lie})}
	\arrow["\simeq", no head, from=1-1, to=2-1]
	\arrow["\simeq", no head, from=1-2, to=2-2]
	\arrow[from=2-1, to=2-2]
	\arrow[from=2-2, to=2-3]
	\arrow[from=1-1, to=1-2]
	\arrow[from=1-2, to=1-3]
	\arrow["\simeq", no head, from=1-3, to=2-3]
\end{tikzcd}\]

After much speculation, this was recently realized by Ching-Salvatore \cite{ching_salvatore}. To complete the program of translating Poincaré-Koszul duality into a more classical form, one must also establish a Koszul self duality result for the right modules $E_M$, the configurations of disks in $M$. In this paper, we construct a compatible diagram of right modules which witnesses the compactly supported Koszul self duality of $E_M$. Combined with work of Ching on Koszul duality for operads in spectra, this gives a construction of a Poincaré-Koszul equivalence for factorization homology of left $E_n$-modules.

 Around the time the self duality of $E_n$ was initially conjectured, Ching noticed the stabilized configuration spaces of framed manifolds could be endowed with the structure of a shifted Lie right module in two ways: one through Goodwillie calculus of the functor $X \rightarrow \Sigma^\infty\operatorname{Map}(M^+,X)$ and one through manifold calculus combined with the hypothesized self duality of $E_n$ \cite{ching}. He conjectured that these two right module structures should coincide. We prove this as a consequences of the self duality of $E_M$.

The main work that occurs in proving the self duality of $E_M$ is finding a way to access the operad equivalence $\Sigma^\infty_+ E_n \simeq s_n K(\Sigma^\infty_+ E_n)$ in a way that minimizes the difficult point-set level constructions of Ching-Salvatore. This operad equivalence is rather technical, involving many models of the $E_n$ operad and sphere operads $\mathrm{CoEnd}(S^n)$. It is reasonable to expect that alterations to the arguments could address the self duality of $E_M$, but we expect this to be prohibitively complex. 

Instead, we interpret Koszul duality in terms of Spivak normal fibrations. Analogously to the classic theory of Poincaré complexes \cite{spivak_1967}, we define the Koszul dualizing fibration of an operad, which has the structure of an operad in parametrized spectra. Using Verdier duality, we show that the Thom complex of the dualizing fibration is a model of the Koszul dual of $O$. We then prove that self duality is equivalent to trivializing the Koszul dualizing fibration as an operad in parametrized spectra. This procedure lifts the homological theory of Poincaré/Koszul operads of \cite[Section 5]{malin2} to operads in spectra, analogous to how Atiyah duality lifts Poincaré duality in (co)homology to a statement in spectra.

A parallel story holds true for right modules over operads. By a combination of the locality of the Koszul dualizing fibration for codimension 0 embeddings and a construction of noncompact Koszul self duality for the right $E_n$-module $E_{\mathbb{R}^n}$,
we are able to explicitly construct a zigzag of maps showing that codimension $0$ subsets of $\mathbb{R}^n$ have compactly supported Koszul self duality, lifting the homological statement of \cite[Theorem 7.8]{malin2}. Ultimately, we extend this result to all tame, framed $n$-manifolds using the theory of Weiss cosheaves.  Of particular note, the argument we present is compatible with any zigzag of maps implementing the Koszul self duality of $E_n$.

\begin{center}
\textbf{Main Results}
\end{center}
Precise statements of the following results can be found in the given sections.

The main ingredient for the compactly supported Koszul self duality of $E_M$ is an interpretation of Koszul duality in terms of Thom complexes $\mathrm{Th}(-)$ of operads and right modules in the category of parametrized spectra. 
\begin{prp}[Propositions \ref{prp:koszulverdieroperad}, \ref{prp:koszulverdiermodule}: Koszul-Verdier duality]
To an operad $O$ and right module pair $(R,A)$ in $(\mathrm{Top},\times)$ we can associate an operad $\xi_O$ and right module $\xi_{(R,A)}$ in $(\mathrm{ParSp},\bar\wedge)$ for which there are compatible equivalences
\[\mathrm{Th}(\xi_O) \xrightarrow{\simeq} K(\Sigma^\infty_+ O)\]
\[\mathrm{Th}(\xi_{(R,A)}) \xrightarrow{\simeq} K(\Sigma^\infty R/A).\]
\end{prp}

The above result quickly leads to the compactly supported Koszul self duality in the case of tame open subsets of $\mathbb{R}^n$, which we then extend to all tame, framed manifolds using the theory of Weiss cosheaves (Definition \ref{dfn:cosheaf}). The statement of Koszul self duality involves operadic suspension $s_n$ (Definition \ref{dfn:opSus}) and right module suspension $s_{(n,n)}$ (Definition \ref{dfn:modSus}).

\begin{thm}[Theorem \ref{thm:selfduality}: Koszul self duality of $E_M$]
    For a framed $n$-manifold $M$, there are equivalences compatible with the self duality of $E_n$:
    \[\Sigma^\infty_+ E_M \simeq s_{(n,n)} K(\Sigma^\infty E_{M^+}).\]
\end{thm}

The self duality of $E_M$ has a range of applications due to the interaction of $E_M$ with various forms of functor calculus. With respect to Goodwillie calculus, it leads directly to a resolution of Ching's conjecture:

\begin{cor}[Corollary \ref{cor:chingconjecture}: Ching's conjecture]
    For a framed $n$-manifold $M$, there is an equivalence of right $ \mathrm{lie}$-modules \[\mathrm{res}_{\mathrm{lie}}(s_{(-n,-n)}\Sigma^\infty_+ E_M) \simeq  \partial_*(\Sigma^\infty \mathrm{Map}_*(M^+,-)).\]
\end{cor}

With respect to embedding calculus, we have a Pontryagin-Thom type construction which extends one point compactification of open embeddings to arbitrary maps of right modules:

\begin{thm}[Theorem \ref{thm:pontry}: Pontryagin-Thom equivalence]
   For framed $n$-manifolds $M,N$ there is a map
    \[\mathrm{Map}^h_{\Sigma^\infty_+ E_n}(\Sigma^\infty_+E_M, \Sigma^\infty_+ E_N) \rightarrow \mathrm{Map}^h_{\Sigma^\infty_+ E_n}(\Sigma^\infty E_{N^+}, \Sigma^\infty E_{M^+}).  \]
    Supposing Conjecture \ref{thm:chingKoszul} that $K:\mathrm{RMod}_{\Sigma^\infty_+ E_n} \rightarrow \mathrm{RMod}_{K(\Sigma^\infty_+ E_n)}$ is an equivalence when restricted to finite type right modules, this map is an equivalence.
\end{thm}
With respect to factorization homology, we achieve a version of Poincaré-Koszul duality \cite{PKayala_francis_2019} for left $\Sigma^\infty_+ E_n$-modules:

\begin{thm}[Theorem \ref{thm:Poincarékoszul}: Poincaré-Koszul duality for left $\Sigma^\infty_+ E_n$-modules]

For a framed $n$-manifold $M$ and a left $\Sigma^\infty_+ E_n$-module $L$ there is an equivalence

\[\int_{\Sigma^\infty_+ E_M} L \xrightarrow{\simeq} \int^{s_{(n,n)}\Sigma^\infty E_{M^+}^\vee} B(L).\]

\end{thm}

\begin{center}
\textbf{Acknowledgements}
\end{center}

Among many people who contributed to this paper through patient conversation, I am especially grateful to Araminta Amabel, Mark Behrens, and Alexander Kupers for their comments on previous versions of this paper. I also thank the anonymous referee for their careful review of this paper. The author acknowledges the support of NSF grant DMS-1547292.

\begin{center}
\textbf{Outline}
\end{center}

Our goal is to remain as concrete as possible without making arguments or constructions prohibitively complicated. Those who have preferences for $\infty$-categorical approaches will have no trouble translating our approach to Verdier duality to that of Lurie \cite[5.5.5]{lurieHA}.

In Section \ref{sec:operad}, we introduce (co)operads and right (co)modules using partial (de)composites, and we recall two models of the $E_n$ operad and the right modules $E_M$, as well as an extension of this construction to the one-point compactification $M^+$. In Section \ref{sec:conjecture}, we recount a conjecture of Ching regarding manifold calculus and Goodwillie calculus. In Section \ref{sec:parsp}, we give a review of parametrized spectra. In Section \ref{sec:verdier}, we prove Verdier duality for the functor of relative sections $ \Gamma^A(-)$. In Section \ref{sec:koszul}, we review the Koszul duality of Ching and Salvatore and construct the Koszul dualizing fibration $\xi_O$ associated to an operad in unpointed spaces. In Section \ref{sec:koszulself}, we prove that codimension $0$ submanifolds of $\mathbb{R}^n$ satisfy noncompact Koszul self duality. In Section \ref{sec:factorization}, we discuss Weiss cosheaves taking value in the category of right modules over an operad. In Section \ref{sec:results}, we prove the main result that $E_M$ is Koszul self dual and give applications.
% In section .., we extend this result to include tangential variants. In section .., we explain the interaction of our results with Poincaré-Koszul duality of \cite{PKayala_francis_2019} and give applications to embedding calculus.

\begin{center}
    \bf{Conventions}
\end{center}

 By $\mathrm{Top}$ or $\mathrm{Top}_*$ we mean a convenient category of topological spaces or pointed topological spaces. All manifolds are assumed tame, i.e. diffeomorphic in an unspecified way to the interior of a compact manifold with boundary. By $\mathrm{Sp}$ we mean the category of orthogonal spectra. We model the category of parametrized spectra $\mathrm{ParSp}$ by orthogonal sequences of ex-spaces, i.e. retractive spaces (explained in Section \ref{sec:parsp}). The homology of an unpointed space is always unreduced, and the homology of a pointed space is always reduced. In the final two sections, we use several models of $\infty$-categories: topological categories, Kan complex enriched categories, topological model categories, simplical model categories, and quasicategories. These are related as follows:
\begin{center}
% https://tikzcd.yichuanshen.de/#N4Igdg9gJgpgziAXAbVABwnAlgFyxMJZABgBpiBdUkANwEMAbAVxiRAB12BbOnACwBOXYAGFeAXwB6wTj35DgAFQhpx4kONLpMufIRQBGclVqMWbWb0HCxOKTO5WFAaTpg1GrSAzY8BIkYGJvTMrIgcjvLCALLQMAy29pZRwADKqTB26pravnpEAEzG1CHm4cnWwACKTHTYiZ65uv4oZEElZmERcpWxsAkS0hUKyqrZJjBQAObwRKAAZgIQXEhkIDgQSEamoRaRlalYYFPZXovLSEXrm4jbDHQARvEACjp++iAMMPM4IB275X2CgE8HEAB8APoOHoKB5YeZYB4CNxZRogc4rW7UDZIADM-zK3SsAGNGMAAHKnBZLTFXHGIfGfR4vN75cICLBTPi-AldCqkhgUpJA4RcOIMDw5dE0pAAFmxNyu9yeDFeeRaIA5XJ5O0Jw2Eh2OVOlF0Q8uuq15exhwhBcHBUP1wDhCKRKMlFHEQA
\begin{tikzcd}
\mathrm{Cat}^{\mathrm{Top}} \arrow[r, "\mathrm{Sing}"]                                                       & \mathrm{Cat}^{\mathrm{Kan}} \arrow[r, "\mathcal{N}"]                                                                          & \mathrm{QuasiCat} \\
\mathrm{ModelCat}^{\mathrm{Top}} \arrow[r, "\mathrm{Sing}"'] \arrow[u, "\mathrm{res}|_{\mathrm{bifibrant}}"] & \mathrm{ModelCat}^{\mathrm{SSet}} \arrow[u, "\mathrm{res}|_{\mathrm{bifibrant}}"] \arrow[ru, "\mathcal{N}^{\mathrm{model}}"'] &                  
\end{tikzcd}
\end{center}

For our purposes, the construction of interesting functors will take place in $\mathrm{ModelCat}^{\mathrm{Top}}$ and the categorical homotopy theory will take place in each of $\mathrm{ModelCat}^\mathrm{SSet},\mathrm{Cat}^\mathrm{Kan},$ and $\mathrm{QuasiCat}$. These models of $\infty$-categories are suitably equivalent. For instance, they have the same theory of mapping spaces and diagram categories \cite[Theorem 2.2.0.1, Proposition 4.2.4.4.]{lurie_2009}. Using these comparisons, we will reduce the statements of our many of our theorems to versions which we verify in $\mathrm{QuasiCat}.$

\section{Operads and right modules of interest}\label{sec:operad}
In this section, we give a brief review of operads and their right modules. For extended discussion, and a comparison to $\circ$-product definitions, we refer to \cite{ching_2005}.
\begin{dfn}
Let $\mathrm{Fin}$ denote the category of nonempty finite sets and bijections. A symmetric sequence in $C$ is a functor $F:\mathrm{Fin} \rightarrow C$  .
\end{dfn}

\noindent We denote the category of symmetric sequences by $\Sigma \mathrm{Seq} (C)$. Given finite sets $I,J$ with $a \in I$, we define  the infinitesimal composite \[I \cup_a J:=I -\{a\} \sqcup J.\] 

\begin{dfn}
An operad in a symmetric monoidal category $(C,\otimes)$ is a symmetric sequence $O$ in $C$ together with morphisms called partial composites for all $a \in I$ and a unit:
\[O(I) \otimes O(J) \rightarrow O(I \cup_a J)\]
\[1_\otimes \rightarrow O(\{*\}).\]  These must satisfy straightforward equivariance and unitality conditions. As well, for $a,a' \in A,b \in B$ we require an associative law corresponding to the identification $(A \circ_a B) \circ_b C = A \circ_a (B \circ_b C)$ and a parallel composition law corresponding to the identification $(A \circ_a B) \circ_{a'} C=(A \circ_{a'} C)\circ_a B$.
\end{dfn}

\begin{dfn}
A cooperad in a symmetric monoidal category $(C,\otimes)$ is a symmetric sequence $P$ in $C$ together with morphisms called partial decomposites for all $a \in I$ and a counit:
\[P(I \cup_a J) \rightarrow P(I) \otimes P(J)\]
\[P(\{*\})\rightarrow 1_\otimes.\]
 These must satisfy straightforward equivariance and counitality conditions. As in the case of operads, these should satisfy coassociativity and parallel cocomposition laws.
\end{dfn}

From now on we restrict to reduced (co)operads, i.e. those cooperads whose underlying symmetric sequence $S$ has  $S(\{*\})$ equal to the unit of $\otimes$. 

\begin{dfn}
A right module $R$ over an operad $O$ in $(C,\otimes)$ is a symmetric sequence $R$ in $C$ with morphisms called partial composites:
\[R(I) \otimes O(J) \rightarrow R(I \cup_a J)\]
for all $a \in I$. These must satisfy straightforward equivariance and unitality conditions, as well as associativity and parallel composition laws.

\end{dfn}

\begin{dfn}
A right comodule $R$ over a cooperad $P$ in $(C,\otimes)$ is a symmetric sequence $R$ in $C$ together with morphisms called partial decomposites:
\[R(I \cup_a J) \rightarrow R(I) \otimes P(J)\]
for all $a \in I$. These must satisfy straightforward equivariance and counitality conditions, as well as coassociativity and parallel cocomposition laws.
\end{dfn}

Of course operads, cooperads, right modules, and right comodules assemble into categories we call $\mathrm{Operad}(C,\otimes),$ $ \mathrm{Cooperad}$, $(C,\otimes),$ $\mathrm{RMod}_O,\mathrm{RComod}_P$ where morphisms respect all the structure. If $C$ has a notion of weak equivalences, then weak equivalences in any of these categories are defined as those which levelwise consist of weak equivalences, and if $C$ is topologically enriched, then all of these categories are also topologically enriched.

It is important to note that all maps respect the symmetric group actions, but the symmetric group actions do not affect whether a given (co)operad/right (co)module map is a weak equivalence. In this sense, the theory of operads is strongly related to the world of Borel $\Sigma_n$-equivariant homotopy theory.

Often we will be dealing with right modules over different operads. There is an obvious compatibility condition for maps of right modules over different operads.

\begin{dfn}
Given a map of (co)operads $f:O \rightarrow P$, and right (co)modules $R,S$ over $O,P$, respectively, we say a map of symmetric sequences $g:R \rightarrow S$ is compatible with $f$ if all the evident diagrams commute. 

More generally, if there is a zigzag of (co)operad maps from $O$ to $P$ through $Q_i$ and a zigzag of symmetric sequence maps from $R$ to $S$ through the right (co)modules $W_i$ over $Q_i$, we say the zigzags are compatible if each symmetric sequence map is compatible with the associated (co)operad map.
\end{dfn}

In this paper, we are primarily interested in specific right modules associated to framed $n$-manifolds. We endow $\mathrm{Emb}(M,N)$ with the compact open $C^\infty$-topology. Fix a framed $n$-manifold $M$, explicitly, an $n$-manifold with a vector bundle isomorphism $\phi_M:TM \cong M \times \mathbb{R}^n$.

\begin{dfn} The Moore isotopy space of $M$, $\mathrm{Iso}(M)$, is the space of maps \[M \times [0,a] \rightarrow \mathrm{GL}(n) \: \mathrm{for} \: 0 \leq a <\infty\] such that $M \times \{0\}$ is mapped to $ \mathrm{Id} \in \mathrm{GL}(n)$. This space is topologized by embedding it into \[\mathrm{Map}(M \times [0,\infty), \mathrm{GL}(n)) \times [0,\infty)\] as the pairs $(f,a)$ where $f(x,t)=f(x,a)$ for all $t\geq a$.
\end{dfn}

 By pointwise composing with $\phi_M$, we can think of the space $\mathrm{Iso}(M)$ as the space of Moore isotopies of the framing of $M$. The descriptor ``Moore'' is used to indicate that the isotopies are not required to have a fixed length.

\begin{dfn}\label{dfn:framedemb}
Let $M,N$ be $n$-manifolds with framings $\phi_M,\phi_N$. The space $\mathrm{Emb}^{\mathrm{fr}}(M,N)$ is the subspace of $\mathrm{Emb}(M,N) \times \mathrm{Iso}(M)$ given by elements \[i: M \rightarrow N, f:M \times [0,a] \rightarrow  \mathrm{GL}(n)\] such that 
\[\phi_M \cdot f|_{M \times \{a\}} = i^*(\phi_N)\]
\end{dfn}

In other words, the space of framed embeddings is the space of embeddings from $M$ to $N$ together with a Moore isotopy of $\phi_M$ to make the embedding strictly preserve the framing. Given framed manifolds $M,M',M''$ there are composition maps
\[\mathrm{Emb}^\mathrm{fr}(M,M') \times \mathrm{Emb}^\mathrm{fr}(M',M'') \rightarrow \mathrm{Emb}^\mathrm{fr}(M,M''')\]
given by composition of embeddings and composition of Moore isotopies. Since the length of Moore isotopies is allowed to vary, this composition is strictly associative.

\begin{dfn}
For $|I| \geq 2$, the operad $E_n$ in $(\mathrm{Top},\times)$ is given by \[E_n(I):= \mathrm{Emb}^{\mathrm{fr}}(\bigsqcup_{i \in I} \mathbb{R}^n, \mathbb{R}^n).\] Operad partial composition is defined by composition of framed embeddings.
\end{dfn}

This model of $E_n$ is equivalent to more classical definitions involving rectilinear embeddings or standard embeddings of disks \cite[Remark 2.10]{ayala_francis_2015}.

\begin{dfn}
For a framed manifold $M$, the right $E_n$-module $E_M$ in $(\mathrm{Top},\times)$ is given by \[E_M(I)=\mathrm{Emb}^\mathrm{fr}(\bigsqcup_{i \in I} \mathbb{R}^n, M).\]
Right module partial composition is defined by composition of framed embeddings.
\end{dfn}

Heuristically, there should be a right $E_n$-module $E_{M^+}$ which consists of the ``framed embeddings of disks into $M^+$ which are allowed to disappear at $\infty$''. Rather than rigorously define this, we instead pass to different models of $E_n,E_M$ which are more amenable to one-point compactifications. 

\begin{dfn}
    If $X$ is a space and $I$ is a finite set, the $I$-labeled configuration space is \[F(X,I):=\{(x_i) \in M^{ I}| x_i = x_j \implies i = j\} .\]
\end{dfn}

From a configuration of disks we may extract a configuration of points by recording the centers of each disks. In order to encode operadic information into such a map, we must ``enhance'' configuration spaces with additional combinatorial information.

% because of how configurations of disks are treated as disks go to $\infty$. For the module $E_{M^+}$, disks can disappear at $\infty$ without the entire configuration getting closer to the basepoint. In a rough sense, one computes the Spanier-Whitehead dual of a framed manifold by dealing with the noncompactness of the manifold. The reason why one might conjecture $E_M$ has compactly supported Koszul self duality is that $B(-)$ addresses the noncompactness of $E_M \simeq F(M,-)$ which arises from points in the configuration approaching each other, while compactifying $M$ itself addresses the noncompactness which arises from points in the configuration approaching $\infty$. However, all of this needs to be done by the addition of a single point which means our definition of $E_{M^+}$ needs a slight adjustment.
%decide whether to include this based on how rest argument goes / level of detail
% The space $E_{M^+}(I)$ has the homotopy type of \[\{(x_i) \in (M^+)^{\wedge I}| x_i = x_j \Rightarrow i = j\} \cup \{(*,*,\dots,*)\}.\]

By an $I$-labelled tree, we mean a tree with:
\begin{itemize}
    \item  a distinguished root,
    \item a bijection between the leaves and $I$,
    \item and satisfying the condition that if $v$ is an internal vertex, i.e. not a leaf or a root, the \textit{outgoing} edges $e(v)$, i.e. edges which are between $v$ and some leaf, should have cardinality at least $2$.
    
\end{itemize}
The Fulton-Macpherson operad, rigorously defined in \cite[Section 3.2]{getzler_jones}, has a heuristic definition as follows:

\begin{dfn}A point in $\mathcal{F}_n(I)$ is represented by an $I$-labeled tree whose root has one outgoing edge together with labels for all nonroot, nonleaf vertices $v$ with values in $F(\mathbb{R}^n, e(v))$. We then quotient the labels of each vertex by translation and positive scaling. The symmetric sequence $\mathcal{F}_n$ is an operad by grafting trees.
\end{dfn}

 We interpret the elements of $\mathcal{F}_n$ as ``infinitesimal configurations" in $\mathbb{R}^n$, modulo the given relations. The label of the vertex adjacent to the root is the ``base configuration'' and, if the tree branches, we imagine that the single point labeled by that edge is actually a configuration in the tangent space, and this may repeat. For a framed manifold $M$, we can construct a right module over this operad, first considered in \cite[Section 5]{markl_1999}.

\begin{dfn}A point in $\mathcal{F}_M(I)$ is represented by an $I$-labeled tree with the root $r$ labeled by $F(M, e(r))$. If $v$ is a nonleaf vertex adjacent to the root, it is labeled by $F(T_p(M), e(v))$ where $p$ is the point in the root configuration labeled by the edge connecting $v$ to the root. Any nonleaf child $v'$ of $v$ is labeled by $F(T_p(M),e(v'))$ for the same $p$, and so on. We then quotient each nonroot, labeled vertex by translation and positive scaling. The symmetric sequence $\mathcal{F}_M$ is a right module over $\mathcal{F}_n$ with partial composition given by grafting trees and using the framing to identify $T_p(M)$ with $\mathbb{R}^n$.
\end{dfn}
We interpret the elements of $\mathcal{F}_{M}$ as ``infinitesimal configurations in $M$''. One important observation is that both $\mathcal{F}_M(I)$ and $\mathcal{F}_n(I)$ are manifolds of dimension $n|I|$ and $n|I|-n-1$, respectively, and the partial composites 
\[\mathcal{F}_n(I) \times \mathcal{F}_n(J) \rightarrow \mathcal{F}_n(I \cup_a J)\]
\[\mathcal{F}_M(I) \times \mathcal{F}_n(J) \rightarrow \mathcal{F}_M(I \cup_a J)\] are inclusions of codimension 0 portions of the boundary. Using the Fulton--MacPherson compactifications, we can formally define a model of $E_{M^+}$ as a right module over $(\mathcal{F}_n)_+ \in \mathrm{Operad}(\mathrm{Top}_*,\wedge)$, the Fulton-MacPherson operad with a disjoint basepoint.

\begin{dfn}
    The right $(\mathcal{F}_n)_+$-module $\mathcal{F}_{M^+}$ in $(\mathrm{Top}_*,\wedge)$ is the levelwise one point compactification of $\mathcal{F}_M$. 

\end{dfn}

\noindent \textbf{Warning:} The space $\mathcal{F}_{M^+}(I)$ does not usually have the homotopy type of any of the following: \[F(M,I)^+,\:  F(M^+,I), \: \mathrm{or}\: \{(x_i) \in (M^+)^I | i \neq j \implies x_i \neq x_j \: \mathrm{or} \: x_i=x_j=* \}.\] Instead it has the homotopy type of \[* \cup \{(x_i) \in (M^+)^{\wedge I} |  x_i = x_j \implies i=j \} \subset (M^+)^{\wedge I}.\] We will always treat the one point compactification $M^+$ as an object distinguished from both manifolds and pointed spaces. Treatments of categories of these so-called ``zero-pointed manifolds'' may be found in \cite{ayala_francis_2019,aroneching2014manifolds}. Finally, we note that $\mathcal{F}_{(\mathbb{R}^n)^+}\not\cong \mathcal{F}_{S^n}$, and this is especially relevant in Lemma \ref{lem:contractible}.

The formal homotopy theory of operads using model categories can be quite complex, but just as for all model categories, it aims to capture the structure of zigzags of morphisms of operads for which the backwards arrows are weak equivalences. We introduce some terminology for zigzags.

\begin{dfn}
A \textbf{zigzag map of operads} is a zigzag of operad maps for which the backwards maps are weak equivalences. Similarly, a \textbf{zigzag equivalence of operads} is a zigzag map of operads which represents an isomorphism in the homotopy category of operads  (or equivalently symmetric sequences) when we invert the backwards maps.
\end{dfn}

\begin{dfn}\label{dfn:zigzag}
A \textbf{zigzag map of right modules} relative to a zigzag map of operads is a zigzag of maps of right modules, compatible with the zigzag of maps of operads, for which the backwards maps are weak equivalences of symmetric sequences. Similarly, a \textbf{zigzag equivalence of right modules} relative to a zigzag equivalence of operads is a zigzag map of right modules which represents an isomorphism in the homotopy category of symmetric sequences when we invert the backwards maps.
\end{dfn}

There is a substantial difference between a zigzag equivalence and a zigzag of equivalences, but in the presence of a model structure the existence of the former implies the existence of the latter.

\section{Ching's conjecture}\label{sec:conjecture}
In Ching's thesis \cite{ching_2005}, he lifted the algebraic theory of Koszul duality to operads in spectra. In particular, he defined a contravariant functor $K:\mathrm{Operad}(\mathrm{Sp},\wedge) \rightarrow \mathrm{Operad}(\mathrm{Sp},\wedge) $, reviewed in Section \ref{sec:koszul}, and showed that the Goodwillie derivatives of the identity are equivalent, as a symmetric sequence, to the spectral Lie operad $\mathrm{lie}:=K(\mathrm{com})$. Later, it was shown that $\partial_*\mathrm{Id}$  forms an operad and the aforementioned equivalence is one of operads. Simultaenously, it was shown the derivatives of functors $\mathrm{Top}_* \rightarrow \mathrm{Sp}$ form right modules over $\partial_* \mathrm{Id}$. These observations were incorporated into the theory of Goodwillie calculus to produce an elegant theory relating Koszul duality to Goodwillie calculus \cite{arone_ching_2011}.

Arone--Ching calculated that as a right module over $\mathrm{lie} \simeq \partial_* \mathrm{Id}$, the derivatives $\partial_*$ of the functor $ \Sigma^\infty \operatorname{Map}_*(X,-)$ are equivalent to the Koszul dual: $K(\Sigma^\infty X^\wedge)$ \cite[Example 17.28]{arone_ching_2011}. Here \[X^\wedge(I):= X^{\wedge I}\] is the right $\mathrm{com}$-module in $(\mathrm{Top}_*,\wedge)$  with partial composites determined by the diagonal. Using the fact that $K:\mathrm{RMod}_O \rightarrow \mathrm{RMod}_{K(O)}$ lifts the Spanier-Whitehead dual of the derived indecomposables (i.e. the quotient by the image of the partial composites), it is straightforward to see that as symmetric sequences $K(X^\wedge)\simeq (X^\wedge/\Delta^{\mathrm{fat}})^\vee$, where the fat diagonal $\Delta^{\mathrm{fat}}$ is the subset of $X^{\wedge I}$ which contains repeated entries (if $|I|=1$, it is defined as $*$). If $X \cong M$ is a compact, framed $n$-manifold, then Atiyah duality yields $(M_+^\wedge/\Delta^{\mathrm{fat}})^\vee \simeq s_{(-n,-n)}\Sigma^\infty_+ F(M,-)$. Hence, up to weak equivalence and a specific type of suspension $s_{(n,n)}$ (Definition \ref{dfn:modSus}), the stabilization of ordered configurations in $M$ is naturally a right module over the derivatives of the identity.

The symmetric sequence of configurations $F(M,-)$ is also equivalent to the configuration space of $n$-disks $E_M$ via radial contraction. Since $M$ is framed, $E_M$ has an action of $E_n$, the little disks operad. Via the self duality of $E_n$
\[K(\Sigma^\infty_+ E_n) \simeq \dots \simeq s_{-n} \Sigma^\infty_+ E_n\]
we can produce a zigzag map $K (\mathrm{com}) \rightarrow s_{-n} \Sigma^\infty_+ E_n$ by taking the Koszul dual of the map $\Sigma^\infty_+ E_n \rightarrow \mathrm{com}$. Restricting along this map gives a right $\mathrm{lie}$-module structure to $s_{(-n,-n)}\Sigma^\infty_+E_M$, and up to weak equivalence and suspension, to $\Sigma^\infty_+F(M,-)$.

Ching conjectured that these two right module structures coincide \cite{ching}. At the time of this conjecture, the self duality of $E_n$ had not been proven and right modules over $E_n$ had not been heavily studied. Recently, the study of the mapping spaces of right modules over $E_n$ has provoked much interest and goes by the name of ``embedding calculus'' \cite{Turchin_2013}. A common situation in embedding calculus is that $M,N$ are framed n-manifolds (often $M$ is a tubular neighborhood of a submanifold of $N$), and we want to understand the derived mapping space of right modules
$\mathrm{Map}^h_{E_n}(E_M,E_N).$
There is map \[\mathrm{Map}^h_{\Sigma^\infty_+ E_n}(\Sigma^\infty_+ E_M,\Sigma^\infty_+ E_N) \rightarrow \mathrm{Map}^h_{\mathrm{com}}(\Sigma^\infty M^\wedge_+, \Sigma^\infty N^\wedge_+)\] obtained by induction along the map $\Sigma^\infty_+ E_n \rightarrow \mathrm{com}$ (which in the model $\mathcal{F}_M$ is simply collapsing the infinitesimal configurations to the fat diagonal). Supposing a noncompact version of Ching's conjecture is true, there is also a restriction map \[\mathrm{Map}^h_{\Sigma^\infty_+ E_n}(\Sigma^\infty_+ E_M,\Sigma^\infty_+ E_N) \rightarrow \mathrm{Map}^h_{\mathrm{lie}}(K(\Sigma^\infty (M^+)^\wedge), K(\Sigma^\infty(N^+)^\wedge)).\] Using the two Koszul dual descriptions of the Goodwillie tower \cite[Example 6.28]{arone_ching_2015}, we can write this in terms of Goodwillie calculus ($P_\infty$) and embedding calculus ($T_\infty$).

\[P_\infty(\Sigma^\infty \mathrm{Map}_*(N^+,-))(M^+) \leftarrow T_\infty(\Sigma^\infty_+ \mathrm{Emb}^{\mathrm{fr}}(-,N))(M) \rightarrow P_{\infty}(\Sigma^\infty \mathrm{Map}_*(M_+,-))(N_+)\]

We will see that the first map can be constructed from a Pontryagin--Thom-esque collapse functor which interchanges the full subcategory of right $\Sigma^\infty_+ E_n$-modules consisting of the $\Sigma^\infty_+ E_M$ with the full subcategory consisting of the $\Sigma^\infty E_{M^+}$; for maps induced by framed embeddings $i:M \rightarrow N$, it can be seen to coincide with the suspension spectrum of the one point compactification $i^+$. For general maps of right modules, it requires the Koszul self duality of $E_M$ and $E_N$ to define.

\begin{remark}
    There is a more general definition of an operad which allows for indexing by the empty set. Our category of operads in spectra embeds into this category of ``unital operads'' by setting $O(\emptyset)=*$, the zero object of spectra. In the setting of Goodwillie calculus, the operads $\mathrm{lie}$ and $\mathrm{com}$ are taken with the assumption that $\mathrm{lie}(\emptyset)=\mathrm{com}(\emptyset)=*$. In contrast, the unstable embedding tower constructed by Weiss uses $ E_n(\emptyset)= *$, which is instead the monoidal unit of $(\mathrm{Top},\times)$, and so produces a tower of a somewhat different flavor.
\end{remark}

% \section{Embedding calculus}

% G- spectrum version

% Trivial loc sys -> trivial action
% pulling back -> inducing group action
% thom trivial -> homotopy orbits trivial (same type of argument where you look at factor to construct map_ \Sigma^\infty_+ EG \wedge X -> X

% functorialty van kampen?

\section{Parametrized spectra }\label{sec:parsp}
In order to give an interpretation of Koszul duality in terms of operad structures on the Spivak normal fibration, we need a suitably functorial construction of the Spivak normal fibration associated to a pair $(X,A)$. Spivak investigated these fibrations and found they were intimately related to relative Poincaré duality for $(X,A)$ \cite{spivak_1967}. His construction involved taking relative embeddings of $(X,A)$ into $(D^n,\partial D^n)$, making the regular neighborhood into a fibration, and restricting to the boundary fibration. Klein found a functorial description of these boundary fibrations in terms of Borel $\Omega X$-spectra \cite{klein_2001}. Here the adjective Borel indicates the objects of the category are spectra with a group action and the weak equivalences are equivariant maps which are underlying weak equivalences of spectra. In \cite{klein_2007}, he supplied a partial translation of this work into the category of parametrized spectra. 

In this section, we review basic notions of parametrized spectra necessary to finish translating Klein's work on the Spivak normal fibration. We use the construction of parametrized orthogonal spectra due to May-Sigurdsson \cite{may_sigurdsson_2006}. These parametrized spectra have a complex homotopy theory, which we will use only a fraction of.

\begin{dfn}
For $X \in \mathrm{Top}$, the category $\mathrm{Top}_X$ of ex-spaces over $X$ has objects given by retractions to $X$, i.e. diagrams 

% https://q.uiver.app/?q=WzAsMixbMCwwLCJaIl0sWzEsMCwiWCJdLFswLDEsInIiLDAseyJjdXJ2ZSI6LTF9XSxbMSwwLCJzIiwwLHsiY3VydmUiOi0xfV1d
\[\begin{tikzcd}
	Z & X
	\arrow["r_Z",bend right=30,  from=1-1, to=1-2]
	\arrow["s_Z",bend right=30, from=1-2, to=1-1]
\end{tikzcd}\]
such that $r_Z \circ s_Z = \mathrm{Id}_X$. We call $Z$ the total space and $X$ the base space. A map of ex-spaces is given by a map of total spaces covering the identity of the base so that the obvious diagrams commute.
\end{dfn}

When referring to ex-spaces, we will usually leave the section implicit and often the retraction as well if there will be no confusion. We will need some basic constructions in the category ex-spaces in order to make constructions in the category of parametrized spectra. Given an ex-space $E \rightarrow X$ and a map $f: Y \rightarrow X$, we shall denote the pullback ex-space over $Y$ by $f^*(E)$. If $f$ is an inclusion, we will sometimes write $E|_Y$. If $Y$ is a point, this is referred to as a fiber.

\begin{dfn}
The external smash product $E \bar\wedge F$ of an ex-space $E$ over $X$ and an ex-space $F$ over $Y$ is given by
\[(E \times F)/ \sim \:\rightarrow X \times Y\]

where $\sim$ is the equivalence relation identifying $E|_x \vee F|_y \subset E|_x \times F|_y$ to a point inside each fiber. There is a natural section making this an ex-space by sending $(x,y)$ to the basepoint of the smash product of the fibers.
\end{dfn}

\begin{dfn}
The internal smash product $E \wedge F$ of ex-spaces $E,F$ over $X$ is \[E \wedge F :=E \bar\wedge F|_{\Delta(X)},\] where $\Delta(X) \subset X \times X$ is the diagonal, i.e $\{(x,x) \in X \times X\}$.
\end{dfn}

\noindent In other words, both internal and external smash products are computed fiberwise, but the internal smash product requires the fibers lie over the same point.

\begin{dfn}
The relative mapping space $\operatorname{Map}^A_X(E,F)$ of two ex-spaces $E,F$ over $X \supset A$  is the subspace of ex-space maps which when restricted to $A$ factor through the section of $F$. Explicitly, if $\operatorname{Map}_X(E,F)$ is the space of ex-space maps, then $\phi \in \operatorname{Map}_X^A(E,F) \subset \operatorname{Map}_X(E,F)$, if and only if \[\phi|_{E|_A}=s_F\circ (r_E)|_{E|_A}.\]
\end{dfn}
\noindent  We use the convention that mapping spaces are always referred to by $\operatorname{Map}(-,-)$ and mapping spectra are always referred to by $F(-,-)$.

\begin{dfn}
The category $\mathrm{Sp}_X$ of orthogonal parametrized spectra over $X$ has objects given by orthogonal sequences of ex-spaces, i.e. for all $n\geq 0$ an ex-space $Z_n$ with an action of $O(n)$ through ex-space automorphisms, together with $O(n) \times O(m)$ equivariant maps: \[\sigma^m_n:Z_n \bar\wedge (S^{m} \rightarrow *) \rightarrow Z_{n+m}\]
These maps should be unital, i.e. $\sigma^0_n =\mathrm{Id}$, and associative. Morphisms of orthogonal parametrized spectra are given by collection of ex-space maps that respect all the given structure.

\end{dfn}

The construction $(-)\bar\wedge (S^{m} \rightarrow *)$ will often be referred to as fiberwise $n$-fold suspension. Our main source of parametrized spectra are the fiberwise suspension spectra of ex-spaces which are defined at the $n$th ex-space by fiberwise $n$-fold suspension. Given an ex-space $E$ over $X$, we denote the fiberwise suspension spectrum $\Sigma^\infty E$; similarly $\Sigma^n E$ will denote fiberwise $n$-fold suspension. Context will always discern suspension spectra from parametrized suspension spectra and $n$-fold suspension from fiberwise $n$-fold suspension. If we merely have a map $Z \rightarrow Y$ with no section, we let $\Sigma^\infty_Y Z$ denote the suspension spectrum of the ex-space $Z \sqcup Y \rightarrow Y$, with the obvious section.

Given an ex-space $E$ over $X$ and an ex-space $F$ over $Y$, there is an evident notion of a generalized morphism from $E$ to $F$. It is a pair of maps $E \rightarrow F$, $X \rightarrow Y$ such that all the obvious squares commute.

\begin{dfn}
The category $\mathrm{ParSp}$ of orthogonal parametrized spectra has objects the parametrized spectra over $X$, as $X$ varies over all spaces. A morphism from a parametrized spectrum $p$ over $X$ to $q$ over $Y$ is a map $f:X\rightarrow Y$ and generalized ex-space maps over $f$, $p(n) \rightarrow q(n)$ which commute with all the structure.
\end{dfn}

\begin{dfn}
The functor $\mathrm{Base}:\mathrm{ParSp}\rightarrow \mathrm{Top}$ is given by taking a morphism of parametrized spectra
% https://tikzcd.yichuanshen.de/#N4Igdg9gJgpgziAXAbVABwnAlgFyxMJZABgBpiBdUkANwEMAbAVxiRDRAF9T1Nd9CKMgEYqtRizYANLj3Z88BIsNKjq9Zq0QgAmrN7ZFg5CsrqJWkAEcuYmFADm8IqABmAJwgBbJCpA4IJAAmc002V30QD28kMn9AxABmbjdPH0Q4gN9qOAALLFccWOoGLDBLKAgmACMGVmpcmDooJDAmBgZqHDosBjZIcsjo9MSuhJCQUsHtSpq6kAamlsQ2jq6evu0B1k4KTiA
\begin{center}
\begin{tikzcd}
p \arrow[r] \arrow[d, Rightarrow] & q \arrow[d, Rightarrow] \\
X \arrow[r, "f"]                  & Y                      
\end{tikzcd}
\end{center}
to the map $f:X \rightarrow Y$.
\end{dfn}

\begin{dfn}
Given a map $g:X \rightarrow Y$ and a parametrized spectrum $p$ over $Y$, let $g^*(p)$ be the parametrized spectrum over $X$ defined by the pullbacks
\[g^*(p)(n):=g^*(p(n) \rightarrow Y).\]
\end{dfn}
A morphism $p$ to $q$ in $\mathrm{ParSp}$ is equivalently a morphism $p \rightarrow f^*(q)$ in $\mathrm{Sp}_X$. In the case $g:X \rightarrow Y$ is an inclusion, we denote $g^*(q)$ by $q|_{X}$. If $X=*$, we refer to this as a fiber, and it has a natural interpretation as an orthogonal spectrum. In general, there is no relation between the fibers over different points of a parametrized spectrum. However, every parametrized spectrum $E$ is levelwise equivalent to a parametrized spectrum $\widebar{E}$ over the same base that does have equivalent fibers when restricted to each path component of the base. Recall the usual functor $\mathrm{PathFib}$ which replaces a map by an equivalent fibration.

\begin{dfn}
For a parametrized spectrum $p$ over $X$, we define $\mathrm{fib}(p)$ by
\[\mathrm{fib}(p)(n)= \mathrm{PathFib}(p(n) \rightarrow X).\]
It has evident $O(n)$-actions, sections, and structure maps.
\end{dfn}

In general, if a parametrized spectrum levelwise consists of fibrations, the homotopy types of the fibers depends only on the path component, and one can extract comparisons between the fibers over $x$ and $y$ via paths from $x$ to $y$. Treating this thought extremely carefully, for $X$ connected and $G=\Omega X$ one can produce an equivalence between $\mathrm{Sp}_X$ and Borel $G$-spectra $\mathrm{Sp}^{BG}$, which is done in \cite{lind_malkiewich_2018}:

\begin{center}
% https://tikzcd.yichuanshen.de/#N4Igdg9gJgpgziAXAbVABwnAlgFyxMJZABgBpiBdUkANwEMAbAVxiRAB12BbOnACwBOXYAAU6AgMpoAvgH0AGiGml0mXPkIoATOSq1GLNpx78hwKdIB6wAEIBxaUr0woAc3hFQAMwEQuSMhAcCCQARmoAIxgwKCQAZkCGOiiGETU8AjYGGC8cEGp6ZlZEDm5eQWEuAmhfLgBPR2VvWrDqYIDI6NjEBOoklLTsDM0QbNz8-SKjMtNhGwgBGAZpTkRjcrMAYwI4HAEmTeHHajg+LHHEAFpQ6QppIA
\begin{tikzcd}
\mathrm{ParSp}_{BG} \arrow[rr, "\mathrm{monodromy}", bend left] &  & \mathrm{Sp}^{BG}. \arrow[ll, "\mathrm{Borel}\:\mathrm{construction}", bend left, shift left]
\end{tikzcd}
\end{center}

We will give formal definitions of the generalized Thom complex and sections of a parametrized spectrum, but one might already guess that under this correspondence they are homotopy orbits and homotopy fixed points, respectively. Indeed, unstably and in unbased spaces, homotopy orbits fiber over $BG$; when we instead work in pointed spaces, homotopy orbits fiber over $BG$, except at a singular point. This corresponds to the collapse point of the generalized, unstable Thom complex. A similar observation can be made in the case of homotopy fixed points and section spaces. But in order to even talk about such comparisons, we need to introduce the homotopy theory of parametrized spectra.

\begin{dfn}
For a parametrized spectrum $p$ over $ X \ni x$, we let $\pi_n(p,x):=\pi_n(\mathrm{fib}(p)|_x)$. 
\end{dfn}

\begin{dfn}
A weak equivalence of parametrized spectra over $X$ is a map of parametrized spectra over $X$ that induces isomorphisms on all homotopy groups with all basepoints.
\end{dfn}

There is a symmetric monoidal product $\bar{\wedge}$ on $\mathrm{ParSp}$ called the external smash product. It is given by the same coequalizer formula as in the nonparametrized case, but replacing spaces with ex-spaces and using the external smash product of ex-spaces. Since restriction to a point $x$ commutes with (co)limits, as it has both a left and a right adjoint, this implies that the external smash product is computed fiberwise, and it has a unit $\Sigma^\infty (S^0 \rightarrow *)=\Sigma^\infty_* *$. 

\begin{dfn}
    For a parametrized spectrum $p$ and a spectrum $Z$, the fiberwise smash product is \[p \bar\wedge Z:= p \bar\wedge (Z \rightarrow *).\]
\end{dfn}

There is also an internal smash product $\wedge$ on $\mathrm{Sp}_X$ which pulls back the external smash product along $\Delta$. It has unit given by $\Sigma^\infty( X \times S^0)=\Sigma^\infty_X X$. This makes the category $\mathrm{Sp}_X$ into a symmetric monoidal category. In fact, this category is closed symmetric monoidal with an internal parametrized mapping spectrum $F(-,-)$ which is fiberwise computed as a mapping spectrum in the category of orthogonal spectra \cite[Theorem 11.2.5]{may_sigurdsson_2006}. This implies that there are evaluation and composition morphisms for $F(-,-)$ which agree fiberwise with those of $\mathrm{Sp}$.

 In the literature, the unit $\Sigma^\infty_X X$ is often referred to by $S^0_X$, we avoid this as it clashes with the convention that a subscript $X$ means adding a disjoint section.  In the same vein, we explicitly write the unit of $(\mathrm{Top}_X,\wedge)$ as $X \sqcup X$.

\begin{dfn}\label{dfn:relativesections}
If $p,q$ are parametrized spectra over $X \supset A$, the spectrum $F^A_X(p,q)$ of relative maps over $X$ is given by
\[F^A_X(p,q)(n):=\operatorname{Map}^A_X(X \sqcup X,F(p,q)(n))\]
If $p= \Sigma^\infty_X X$, we denote this $\Gamma^A(q)$, the relative sections.
\end{dfn}

Note that since $\mathrm{ParSp}_X$ is closed symmetric monoidal, $F(\Sigma^\infty_X X , q)\cong q$, so $\Gamma^A(q)$ coincides with the spectrum which at level $n$ is $\operatorname{Map}^A_X(X \sqcup X,q(n))$, i.e. the sections of $q(n)$. In terms of base change functors (i.e. the adjoints $f_{!},f_*$ of restriction $f^*$ \cite[Section 11.4]{may_sigurdsson_2006}), if $c$ denotes the constant map to a point there is an equality $F_X(p,q)= c_* F(p,q)$.  We note that $\Gamma^A (\Sigma^\infty_X X)$ is isomorphic to \[(\Sigma^\infty X/A)^\vee := F(\Sigma^\infty X/A,S^0).\] 

\begin{prp}\label{prp:comp}
If $p,q,r$ are parametrized spectra over $X \supset A$, there is a natural composition map

\[F^A_X(p,q) \wedge F_X(q,r) \rightarrow F^A_X(p,r).\]

\end{prp}

\begin{proof}
It suffices to examine how taking levelwise section spaces interacts with internal smash products of parametrized spectra since there is a composition map $F(p,q) \wedge F(q,r) \rightarrow F(p,r)$ since $\mathrm{Sp}_X$ is closed symmetric monoidal. In particular, we want to construct a map \[c_*(p) \wedge c_*(q) \rightarrow c_*(p \wedge q).\] 
By adjointness, this is equivalent to a map of parametrized spectra
\[ c^* (c_*(p) \wedge c_*(q)) \rightarrow p \wedge q\]

The domain is the ``trivial'' parametrized spectrum with fiber $c_*(p) \wedge c_*(q)$ over every point. To construct the above map, for every pair $(c_*(p) \wedge c_*(q),x)$ we should construct a map of spectra to $p|_x \wedge q|_x$. This map is given by taking the smash product of the maps 
\[c_*(p) \rightarrow p|_x\]
\[c_*(q)\rightarrow q|_x\]
given by evaluating the section at $x$. In other words, the original map \[c_*(p) \wedge c_*(q) \rightarrow c_*(p \wedge q)\] is computed by pointwise smashing the sections. The relativity of the statement follows immediately from this description.
\end{proof}

%Im 70 percent sure that the external functors are literally defined by smash products

\begin{dfn}
Suppose we have a parametrized spectrum $p$ over $X$ and a parametrized spectrum $q$ over $X \times X$, and $A \subset X$. For $x\in X$, let $q^x$ denote the restriction $q|_{\{x\} \times X}$; we define parametrized spectra over $X$ in terms of their fibers:

\[\operatorname{F}^A_{X}(p,q)|_{x}:= \operatorname{F}^A_{X}(p,q^{x})\]
\[\operatorname{F}^A_{X}(q,p)|_{x}:= \operatorname{F}^A_{X}(q^{x},p)\]

\end{dfn}
In the case $A$ is empty, we leave it off the notation. If the domain is $\Sigma^\infty_X X$, we denote this by $\Gamma^A(-)$ and distinguish it from Definition \ref{dfn:relativesections} based on whether the parametrized spectrum is over $X$ or $X \times X$. 

\begin{remark}
In comparison to the category of Borel $G$-spectra, parametrized spectra over $X \times X$ play the role of spectra with a two sided action of $G$. The fact that the ``parametrized maps of parametrized spectra'' construction above leaves us with a parametrized spectrum over $X$ corresponds to the fact that spectrum of equivariant maps into or out of a two sided $G$-spectrum still forms a $G$-spectrum. If one was to do substantial homotopy theory with this construction, it would be helpful to also consider the slices $X \times \{x\}$. Passing back and forth between spectra over $\{x\} \times X$ and spectra over $X \times \{x\}$ is analogous to passing between $G$-spectra and $G^\mathrm{op}$-spectra.
\end{remark}
An ex-space is f-cofibrant if the distinguished section is an NDR (neighborhood deformation retract) such that the deformation retraction respects the fibers, and a parametrized spectrum is level-f-cofibrant if each of its ex-spaces is f-cofibrant.

 \begin{lem}\label{lem:cofib}
 Given a map $E \rightarrow X \times X$, the parametrized spectrum $\Gamma^A(\Sigma^\infty_{X \times X} E)$ is level-f-cofibrant.
 \end{lem}

 \begin{proof}
 Of course the zeroth ex-space is f-cofibrant. Now assume $n>0$. Let $D \subset \Gamma^A(\Sigma^\infty_{X \times X} E)(n)$ denote the subspace which consist of sections which only take on values which have a suspension coordinate in $[0,.1] \cup [.9,1]$. This neighborhood witnesses the image of the preferred section $X \rightarrow \Gamma^A(\Sigma^\infty_{X \times X} E)(n)$ as a fiberwise neighborhood deformation retract.
 \end{proof}

We refer to the weaker property of all the distinguished sections being Hurewicz cofibrations as being level-h-cofibrant. Similarly, if all the retractions are Hurewicz fibrations, we call the spectrum level-h-fibrant. In practice, if one shows a spectrum is level-h-cofibrant by hand, almost certainly one has actually shown the stronger statement that it is level-f-cofibrant.

\begin{dfn}
 For a parametrized spectrum $p$ over $X$, the generalized Thom complex $\mathrm{Th}(p)$ is the spectrum given by $\mathrm{Th}(p)(n):=p(n)/\mathrm{im}(s_{p(n)}).$
\end{dfn}

To construct a map out of a Thom complex of $p$ into a spectrum $H$, it suffices to supply a family of spectrum maps out of all the fibers, i.e. a map of parametrized spectra from $p$ to $X \bar\wedge H:= \Sigma^\infty_X X \bar\wedge (H \rightarrow *)$. Recall our name for the constant map $c:X \rightarrow *$. In terms of base change adjunctions, we have that $\mathrm{Th}(p)= c_!(p)$.
It is not obvious, but Thom complexes preserve weak equivalences of level-h-cofibrant (and thus level-f-cofibrant) parametrized spectra \cite[Proposition 6.4.5]{malkbasic}

The following is found in \cite[Proposition 5.6.1]{malkbasic}:
\begin{prp}\label{prp:wedge}
The natural map $ \mathrm{Th} (p\bar\wedge q) \rightarrow \mathrm{Th}(p) \wedge \mathrm{Th}(q) $, induced by the inclusions $p|_x \wedge p|_y \rightarrow \mathrm{Th} (p\bar\wedge q)$ of fibers into the Thom complex, is an isomorphism.
\end{prp}

Since the category $(\mathrm{ParSp},\bar\wedge)$ is symmetric monoidal, we may consider operads in $\mathrm{ParSp}$ using our usual partial composite definition. We make a few simplifying assumptions:

\begin{dfn}\label{dfn:resOp}
The category $\mathrm{ResOp}$ is the restricted category of operads in $\mathrm{ParSp}$. It has objects given by operads $O$ in $(\mathrm{ParSp},\bar\wedge)$ consisting of level-h-cofibrant parametrized spectra, and its morphisms are given by operad morphisms which levelwise are contained in $\mathrm{Sp}_{\mathrm{Base}(O(I))}$. A weak equivalence is an operad map which for each $I$ is a weak equivalence of parametrized spectra over $\mathrm{Base}(O(I))$.
\end{dfn}

\begin{dfn}\label{dfn:resMod}
The category $\mathrm{ResMod_O}$ is the restricted category of right modules over $O \in \mathrm{ResOp}$. It has objects given by right modules over $O$ in $\mathrm{ParSp}$ consisting of level-h-cofibrant parametrized spectra, and its morphisms are given by right module morphisms which levelwise are contained in $\mathrm{Sp}_{\mathrm{Base}(R(I))}$. A weak equivalence is a right module map which for all $I$ is a weak equivalence of parametrized spectra over $\mathrm{Base}(R)(I)$. 
\end{dfn}
From now on, all maps of operads and right modules in parametrized spectra are assumed to lie in $\mathrm{ResOp},\mathrm{ResMod}_O$. Note that any operad $O$ in $\mathrm{ResOp}$ has an associated operad $\mathrm{Base}(O)$ in $(\mathrm{Top},\times)$. The restricted morphisms cover the identity of $\mathrm{Base}(O)$. The same is true for restricted right modules.

\begin{dfn} \label{dfn:thomOp}
For an operad $O$ in $\mathrm{ResOp}$, $\mathrm{Th}(O)$ is the operad in $(\mathrm{Sp}_*,\wedge)$ given by \[\mathrm{Th}(O)(I) :=\mathrm{Th}(O(I)) .\]
The partial composites are given by taking Thom complexes of the partial composites of $O$ and applying the inverse of the isomorphism of Proposition \ref{prp:wedge}.
\end{dfn}

\begin{dfn}\label{dfn:thomMod}
For a right module $R$ in $\mathrm{ResMod}_O$, $\mathrm{Th}(R)$ is the right module over $\mathrm{Th}(O)$ given by \[\mathrm{Th}(R)(I) :=\mathrm{Th}(R(I)) .\]
\end{dfn}

% We assume we have a Quillen equivalence between $\mathrm{Sp}_X$ and $\Omega X - \mathrm{Sp}$ such that there is a commutative diagram:

% % https://q.uiver.app/?q=WzAsMyxbMCwxLCJcXG1hdGhybXtTcH1fWCJdLFsxLDEsIlxcT21lZ2EgWCAtIFxcbWF0aHJte1NwfSJdLFswLDAsIlxcbWF0aHJte1NwfSJdLFswLDIsIlxcR2FtbWFfQSJdLFsxLDIsIigtKV57aCBcXE9tZWdhIFggLSBoXFxPbWVnYSBBfSIsMl0sWzAsMV0sWzEsMCwiIiwxLHsib2Zmc2V0IjotMn1dXQ==
% \[\begin{tikzcd}
% 	{\mathrm{Sp}} \\
% 	{\mathrm{Sp}_X} & {\Omega X - \mathrm{Sp}}
% 	\arrow["{\Gamma_A}", from=2-1, to=1-1]
% 	\arrow["{(-)^{h \Omega X - h\Omega A}}"', from=2-2, to=1-1]
% 	\arrow[from=2-1, to=2-2]
% 	\arrow[shift left=2, from=2-2, to=2-1]
% \end{tikzcd}\]

% Given a parametrized spectrum $Z$, we define the generalized Thom complex $X^Z$ to be the pullback of homotopy orbits under this equivalence. Equivalently, it can be defined as the stabilization of the section collapse map for the category of spaces over $X$ with a section.

\section{Verdier Duality}\label{sec:verdier}

Verdier duality is the process of encoding a colimit preserving functor on a category $C$ in terms of an object $c \in C$. Topological Verdier duality, in its simplest form, has many formulations: in terms of spectra valued presheaves on $\infty$-groupoids \cite{lurieVerd}, locally constant sheaves of spectra  \cite{Volpe}\cite[Section 5.5.5]{lurieHA}, and Borel $G$-spectra. In this last setting, Klein develops much of the theory of Verdier duality in the case of the functor \[\mathrm{hofib}((-)^{hG} \rightarrow (-)^{hH}):\mathrm{Sp}^{BG} \rightarrow \mathrm{Sp}\]
associated to a topological group map $H \rightarrow G$ \cite{klein_2001,klein_2007}.

\begin{dfn}
Let $P$ be the parametrized spectrum over $X \times X$ given by \[ \Sigma^\infty_{X \times X} (\mathrm{ev}_0 \times \mathrm{ev}_1:\mathrm{Path}(X) \rightarrow X \times X).\]
In other words, $P|_{(x,y)}= \Sigma^\infty_+ \mathrm{Path}(x,y)$
\end{dfn}

\begin{lem}\label{lem:free}

For a parametrized spectrum $q$ over $X$ which is level-h-fibrant, there is an equivalence of parametrized spectra \[F_{X}(P,q) \xrightarrow{\simeq} q\]
which on the fiber over $x$ is induced by evaluation at the constant path at $x$. 
\end{lem}

\begin{proof}
% Modulo figuring out the correct fibrancy conditions, one can appeal to the straightforward observation that under the equivalence of parametrized spectra over $X \times X$ with two sided naive $\Omega X$ spectra, $p$ is equivalent to $\Sigma^\infty_+ \Omega X$, whence the statement becomes $\operatorname{Map}_{\Omega X}(\Sigma^\infty_+ \Omega X, Z) \simeq Z$ which is true because $\Sigma^\infty_+ \Omega X$ is a rank 1 free module. This is essentially observed by Klein.

Since $P$ is a fiberwise suspension spectrum,  $F_X(P,q)$ can be computed in an easy way. Fiberwise it may be written as the spectrum which has its $n$th space \[\mathrm{Map}_X( X \sqcup \mathrm{Path}(x,-), q(n)).\] Since $q$ is level-h-fibrant, the question is implied by showing evaluation at the constant path is an equivalence \[\mathrm{Map}_X( X \sqcup \mathrm{Path}(x,-), q(n))\rightarrow q(n)|_x.\] 

 \noindent Since we are working stably, it is no loss to assume the fibers of $q$ are connected, and so $q(n)$ is connected for $n$ large enough. The homotopy fiber of evaluation consists of ex-space maps together with a path in $q(n)|_x$ from the image of the constant path to the basepoint of $q(n)|_x$.

 Now consider the fibration $\Phi$ over the total space of $\mathrm{Path}(x,-)$ which has fiber over a path $\gamma: x \rightarrow x'$ equal to the space of paths lifting $\gamma$ in $q(n)$. Since the base of this fibration is contractible, the space of sections has the based homotopy type of any fiber, in particular the fiber over the constant path $\mathrm{const_X}$, and the space of nullhomotopies of a fixed section $S$ has the homotopy type of the path space $\mathrm{Path}(*,S(\mathrm{const_x}) )$.
%loop -> path
 For $t \in [0,1]$ and a path $\gamma$, define the path $\gamma_{\leq t}(r)= \gamma(\mathrm{min}(r,t))$.
Then if $(f,\alpha)$ is a point in the homotopy fiber of our evaluation map, the ex-space map $f$ determines a section of $\Phi$ by lifting the path $\gamma$ to the path 
 $t \rightarrow f(\gamma_{\leq t})$.
The path $\alpha$ in $q(n)|_x$ then determines a nullhomotopy of sections of $\Phi$ by the above argument which demonstrates that the homotopy fiber is contractible.
\end{proof}
%make sure defined wedging with a spectrum; gepner blumberg
\begin{lem}\label{lem:canSec}
If $Z$ is a spectrum, the equivalence \[F_{X}(P,\Sigma^\infty_X X \bar \wedge Z ) \xrightarrow{\simeq} \Sigma^\infty_X X \bar \wedge Z  \] has a natural section.
\end{lem}

\begin{proof}
There is a natural choice of path-lifting in a product: the constant one. With this choice the sections (expressed here fiberwise)
\[Z(n) \rightarrow \mathrm{Map}_X(X \sqcup \mathrm{Path}(x,-),X \times Z(n)) \]
are natural and commute with the structure maps of $F_{X}(P,\Sigma^\infty_X X \bar \wedge Z) $.

\end{proof}

\begin{dfn}
The Verdier dualizing parametrized spectrum, or Verdier dual, $P(X,A)$ of a pair $(X,A)$ is the parametrized spectrum over $X$ given by $\Gamma^A(P)$. 
\end{dfn}
\noindent

We are ready to state Verdier duality, however, because point-set topology interacts rather poorly with the internal smash product, we will have to derive the internal smash product, as well as the Thom complex, though this is less severe and simply requires taking a cone of the preferred section rather than quotienting. For techniques to derive the internal smash product see  \cite[Page 98]{malkbasic}.

% Recall that a parametrized spectrum is level-h-fibrant if all its retractions are fibrations. There is another notion of being a ``freely f-cofibrant'' which is stronger than being level-f-cofibrant. By definition, this class contains the parametrized suspension spectra of ex-spaces for which the distinguished section is f-cofibrant.

\begin{thm}[Verdier Duality]
Suppose $A \rightarrow X$ is a cofibration such that $A$ and $X$ each have the homotopy type of a compact CW complex, then for any $q \in \mathrm{Sp}_X$ there is a zigzag,  \[\mathrm{Th}( P(X,A) \wedge q) \leftarrow  \mathrm{Th}(  P(X,A)\wedge F_{X}(P,q)) \rightarrow  \Gamma^A( q)\] 
where the rightmost map is given on the fiber over $x \in X$ by the composition
\[ F_X^A(\Sigma^\infty_X X ,\Sigma^\infty_X \mathrm{Path}(x,-)) \wedge F_X(\Sigma^\infty_X \mathrm{Path}(x,-),q) \rightarrow F_X^A(\Sigma^\infty_X X,q)= \Gamma^A(q)\]
If $q$ is level-h-fibrant, then after deriving $\wedge,\mathrm{Th}$ these are equivalences.
\end{thm}

\begin{proof}
The construction of the right hand map occurs in Proposition \ref{prp:comp}. Let $H \leq G$ be topological groupoids, let $bG,bH$ denote the associated topological categories, and let $BG,BH$ their topological realizations. Lastly, assume $BG \leq BH$ is a cofibration. We list a few correspondences under the equivalence $\mathrm{Sp}_{BG} \simeq \mathrm{Sp}^{BG}=\mathrm{Fun}(bG,\mathrm{Sp}).$ 

\begin{enumerate}
  \item The derived internal smash product of parametrized spectra over $BG$ corresponds to the derived smash product of Borel $G$-spectra.
  \item The derived Thom complex corresponds with homotopy $G$-orbits, i.e. the homotopy colimit of the diagram of spectra.
  \item The derived sections relative to $BH$,  corresponds to \[\mathrm{hofib}((-)^{hG} \rightarrow (-)^{hH}).\]
  Here the indicated homotopy fixed points are computed as the homotopy limit of the diagram of spectra.
  \item The construction $P(X,A)$ corresponds to the Borel $G$-spectrum  \[\mathrm{hofib}(\Sigma^\infty_+ G^{hG} \rightarrow \Sigma^\infty_+ G^{hH}). \]
\end{enumerate}

The first follows from the fact that smash products are computed fiberwise. The second follows from the explicit construction of $Z_{hG}$ as $(Z \wedge \Sigma^\infty_+ EG)_G$ and observing that this is the Thom complex of the Borel construction applied to $Z$. The third is essentially dual to the second. These are worked out in \cite[Example 8.1.5]{malkbasic}. Note the derived relative sections are computed simply by taking a level-h-fibrant replacement since (i) the inclusion $BH \rightarrow BG$ is assumed to be a cofibration and (2) derived mapping spectra out of a fiberwise CW suspension spectrum like $\Sigma^\infty_{BG}BG$ are computed by taking level-h-fibrant replacements of the codomain, since it is easily checked weak equivalences between such parametrized spectra are preserved.  The fourth (proven by Klein in the case $A=\emptyset$  \cite[ Theorem 5.6]{klein_2007}) follows from (3) and Lemma  \ref{lem:free}.

Let $G$ denote the topological path groupoid of $X$ and $H$ the topological path groupoid of $A$. With these observations, under the equivalence of categories $\mathrm{Sp}_{X} \simeq \mathrm{Sp}^{BG}$, the composition map described above becomes the same as Klein's \cite[Remark 3.1]{klein_2001}\cite[Proof of Proposition 4.1]{klein_2007}.\footnote{Strictly speaking Klein works in the setting of groups rather than groupoids, but the generalization to groupoids still holds (cf. \cite[Page 3]{lurieVerd}).} The compactness assumptions on $(X,A)$ then imply that Klein's conditions for this map to be an equivalence hold.
\end{proof}

The necessity of deriving these functors will make it difficult to study the interaction of Verdier duality and parametrized operads. However, we are interested in a simple example for which we can both avoid using a zigzag and ignore derived functors.

\begin{cor}\label{cor:relativeswduality}
If $A \rightarrow X$ is a cofibration such that $A$ and $X$ each have the homotopy type of a compact CW complex, there is a canonical equivalence \[\mathrm{Th}(P(X,A)) \xrightarrow{\simeq} \Sigma^\infty (X/A)^\vee.\]
\end{cor}

\begin{proof}
Let  $q=\Sigma^\infty_X X$. If we compose the section provided by Lemma \ref{lem:canSec} with the righthand map of underived Verdier duality, we get a map which on fibers is
\[ F_X^A(\Sigma^\infty_X X,\Sigma^\infty_X \mathrm{Path}(x,-)) \wedge S^0 = F_X^A(\Sigma^\infty_X X,\Sigma^\infty_X \mathrm{Path}(x,-)) \rightarrow F_X^A(\Sigma^\infty_X X,\Sigma^\infty_X X ) \]
Since we have avoided taking internal smash products and only took Thom complexes of a level-h-cofibrant spectrum by Lemma \ref{lem:cofib}, in the homotopy category this map agrees with what we would obtain by taking derived Verdier duality and inverting the first map. Hence, it is an equivalence.

\end{proof}

The equivalence $\mathrm{Th}(P(X,A))\xrightarrow{\simeq} (\Sigma^\infty X/A)^\vee$ has a very simple description. On the fiber over $x\in X$, it is determined by the composition \[F_X^A(\Sigma^\infty_X X,\Sigma^\infty_X \mathrm{Path}(x,-)) \wedge S^0 \rightarrow F_X^A(\Sigma^\infty_X X,\Sigma^\infty_X X )\]
where $S^0$ is picking out the map of parametrized spectra that Lemma \ref{lem:canSec} produces.
This map is simply the fiberwise suspension spectrum of the fiberwise collapse map in $\mathrm{Map}_X(X \sqcup \mathrm{Path}(x,-), X \sqcup X)$. So in this case, Verdier duality is given by composition with fiberwise collapse.

\begin{remark}
    Although the parametrized spectrum $\Sigma^\infty_X X$ has the simplest implementation of Verdier duality, one cannot disregard the other cases. In the Borel category a similar description of Verdier duality holds for trivial $G$-spectra, but it is only an equivalence when $BG$ is compact. This is because one first demonstrates that Verdier duality is an equivalence for free $G$-spectra, then proceeds by an induction which requires the compactness of $BG$. \cite[Section 3]{klein_2001}.
\end{remark}
The Verdier dualizing parametrized spectrum $P(X,A)$ captures a great deal of information about Poincaré duality. Let $k$ be a field. 
\begin{prp}\label{prp:Poincaréduality}
If $A \rightarrow X$ is a cofibration of spaces with the homotopy type of compact CW complexes such that $P(X,A)$ has fiber equivalent to $S^{-n}$ and there is a global choice of orientation of the fibers $H_{-n}(P(X,A)|_x;k)$, then $H_n(X,A;k)$ has a fundamental class $[\alpha]$ such that
\[H^*(X,A;k) \xrightarrow{\cap \alpha} H_{n-*}(X;k)\]
is an isomorphism.
\end{prp}

\begin{proof}
This follows from Verdier duality and the Thom isomorphism theorem for parametrized spectra with spherical fibers\cite[Theorem 20.5.8]{may_sigurdsson_2006}.
\end{proof}

The Verdier dual $P(X,A)$ has a canonical description which leads one to believe it might be functorial. We need a precise understanding of its functoriality in order to study the interaction with operads.

\begin{dfn}
The symmetric monoidal category $(\mathrm{Top}_{\subset},\times)$ has objects given by cofibrations of spaces with the homotopy type of compact CW complexes. The morphisms \[(A \rightarrow X) \rightarrow (B \rightarrow Y)\] are maps $f:X \rightarrow Y$ satisfying:

\begin{enumerate}
\item The map $f$ is a quotient map onto its image.
    \item The restriction $f|_{X - A}$ is an open embedding.
  \item $f(\mathrm{closure}(X-A)) = \mathrm{closure}(f(X-A))$.
  \item $f^{-1}(B) \subset A$.
\end{enumerate}
The pushout product $\times$ on $\mathrm{Top}_\subset$ is given by
\[(X,A) \times (Y,B):= (X \times Y, (X \times B) \cup (A \times Y)).\]
\end{dfn}

% \noindent There is a pointed variant:
% \begin{dfn}
% Consider the category $(\mathrm{Top}_{*,\subset},\times)$ with objects given by pointed cofibrations of pointed spaces with the homotopy type of compact CW complexes morphisms \[(A \rightarrow X) \rightarrow (B \rightarrow Y)\] are pointed cofibrations from $X$ to $Y$ which are i) isomorphic to the inclusion of the closure of an open set and ii) have $X \cap B \subset A$.

% \noindent We define a symmetric monoidal product $\wedge$
% \[(X,A) \wedge (Y,B):= (X \wedge Y, (X \wedge B) \vee (A \wedge Y))\]
% \end{dfn}

A morphism $f:(X,A) \rightarrow (Y,B)$ in $\mathrm{Top}_\subset$ decomposes $Y$ into two closed subsets: $\mathrm{closure}(f(X-A)))$ and $Y -f(X-A)$. This first subset is the closure of the embedded copy of $X-A$ which is also naturally a quotient of $X$. Using this observation, we can extend relative Verdier duality to a functor \[P:\mathrm{Top}_\subset \rightarrow \mathrm{ParSp}.\] Associated to a morphism $(X,A) \xrightarrow{f} (Y,B)$, the map $P(X,A)_x \rightarrow P(Y,B)_{f(x)}$ is given by sending a section $s$ relative to $A$ of $\Sigma^\infty_+\mathrm{Path}(x,-)$ to the section relative to $B$ of $\Sigma^\infty_+ \mathrm{Path}(f(x),-)$ given by $s$ on the embedded image of $X-A$ and factoring through the zero section on $Y - f(X-A) \supset B$. The conditions (2)--(4) allows this section to be well-defined and relative to $B$ as a set-theoretic function, while (1)+(3) imply it is actually continuous.

The conditions on $\mathrm{Top}_\subset$ also allow us to have a well-defined collapse map associated to $f$:
\[Y/B \rightarrow X/A\]
\[y \mapsto f^{-1}(y).\]

Via the collapse map, the assignment $(X,A)\mapsto (\Sigma^\infty X/A) ^\vee$ can be made covariantly functorial on $\mathrm{Top}_\subset$. By Corollary \ref{cor:relativeswduality}, we understand the relation of these two functors.
% In fact, if $A \rightarrow X, B\rightarrow X$ are cofibrations of spaces equivalent to finite $CW$ complexes, the  map can be seen to be equivalence as a consequence of the fact the parametrized spectrum that implements Verdier duality is unique \cite{lurieVerd}. By our choice of morphisms in $\mathrm{Top}_\subset$ we have:

\begin{prp}\label{prp:functor}
Verdier duality gives a natural equivalence of functors \[(X,A) \rightarrow \mathrm{Th} ( P(X,A)) \] \[\Downarrow\] \[(X,A) \rightarrow  (\Sigma^\infty X/A)^\vee.\]
\end{prp}

Given an $x \in X$, we denote the path component of $x$ by $X_x$.

\begin{cor}\label{cor:spherical}
Let $f:(X,A) \rightarrow (Y,B)$ be a morphism in $\mathrm{Top}_\subset$ with $Y$ path-connected. Suppose the fibers of $P(X,A)$ and $P(Y,B)$ are spherical of dimension $-n$ and oriented with respect to the field $k$. Then $f$ induces an isomorphism
\[H_{-n}(P(X,A)|_x;k)
\xrightarrow{\cong} H_{-n}(P(Y,B)|_{f(x)};k),\]
if and only if it induces an isomorphism
\[\bar{H}_n(Y/B;k) \xrightarrow{\cong} \bar{H}_n(X_x/(A \cap X_x);k)\]
\end{cor}

\begin{proof}
Choose orientations of the fibers. By Proposition \ref{prp:Poincaréduality}, this fixes isomorphisms \[H^{-n}(\mathrm{Th}(P(X_x,A \cap X_x));k)\cong H^{-n}(\mathrm{Th}(P(Y,B));k) \cong \bar{H}_n(X_x/(A \cap X_x);k)\cong\bar{H}_n(Y/B;k) \cong k.\] 
Under these identifications it is straightforward to check that degree is preserved up to a sign.
\end{proof}

Finally, we observe that there is a nice interaction between the symmetric monoidal products of $(\mathrm{ParSp},\bar\wedge)$ and $(\mathrm{Top}_\subset,\times)$.

\begin{dfn}\label{dfn:lax}
For $(X,A),(Y,B) \in (\mathrm{Top}_\subset,\times)$, there is as natural map
\[P(X,A) \bar\wedge P(Y,B) \rightarrow P((X,A)\times (Y,B)).\]

Fiberwise the map is given by

%    \\ 
%  \\
% 

\begin{equation*} \label{eq1}
\begin{split}
P(X,A) \bar\wedge P(Y,B)|_{(x,y)} & \cong \Gamma^A(\Sigma^\infty_X (\mathrm{Path}(x,-) ) \wedge \Gamma^B(\Sigma^\infty_Y (\mathrm{Path}(y,-) ) \\
 & \rightarrow \Gamma^{(A \times Y) \cup (X \times B)}(\Sigma^\infty_X (\mathrm{Path}(x,-)) \bar \wedge \Sigma^\infty_Y (\mathrm{Path}(y,-) )\\
 & \cong \Gamma^{(A \times Y) \cup (X \times B)}(\Sigma^\infty_{X \times Y}(\mathrm{Path}((x,y),-)))
\end{split}
\end{equation*}

\end{dfn}

\section{Koszul-Verdier duality of operads and right modules} \label{sec:koszul}
There are two different constructions of Koszul duality for operads in spectra. One is due to Ching and Salvatore, which we will describe in this section and refer to by $K(-)$, and the other is $\mathrm{bar}(-)^\vee$ which is due to the general theory of bar-cobar duality of associative algebras due to Lurie.\footnote{In the case of the $E_n$ operad this has not, to our knowledge, been related to Lurie's theory of bar-cobar duality for $E_n$-algebras or the bar-cobar duality of Ayala-Francis.} Brantner has given a rigorous comparison of these two constructions on the level of operads in $\mathrm{Ho}(\mathrm{Sp})$ \cite[Proposition 5.4.19]{brantner}, but we are not familiar with any comparison as operads in spectra. Currently, the Koszul self duality of $E_n$ is only known in regards to the point-set construction $K(-)$.

We will recall some facts about various W-constructions and bar constructions which can be found (with slight variation) in \cite{malin2}. The origin of these constructions can be found throughout \cite{boardman_vogt_1973,ching_2005, salvatore_1998} with explicit methods to topologize the weighted/labeled trees. We are interested in both pointed and unpointed versions of these constructions, and so to save space we will recover unpointed definitions from their pointed counterparts. The goal of this section is to give a construction of the functor $K(-)$ in terms of generalized Thom complexes and investigate its consequences. All operads are assumed to be reduced.

As before, an $I$-labelled tree is a tree with:
\begin{itemize}
    \item  a distinguished root,
    \item a bijection between the leaves and $I$,
    \item and satisfying the condition that if $v$ is an internal vertex, i.e. not a leaf or a root, the \textit{outgoing} edges $e(v)$, i.e. edges which are between $v$ and some leaf, should have cardinality at least $2$.
    \end{itemize}

\begin{dfn}[The $W$-construction of an operad]
Given an operad O in $(\mathrm{Top}_*,\wedge)$, let $T(O)(I)$ denote the space
of $I$-labeled, rooted trees with the property that the root has a single outgoing edge and an internal vertex $v$ is labeled by $O(e(v))$, and nonroot and nonleaf adjacent edges are
labeled by an element of $[0, \infty]$; any tree containing a vertex labeled by the basepoint is collapsed to a single point. We let $W(O)(I)$ denotes the quotient of $T(O)(I)$ by the relation that any length 0 edge can be collapsed
by applying operadic partial composition. The collection of $W(O)(I)$ is an operad 
$W(O)$ by grafting trees via length $\infty$ edges.
\end{dfn}

\begin{dfn}[The $W_{[0,\infty]}$-construction of an operad]
The thickened $W$-construction $W_{[0,\infty]}(O)$ is given by
\[W_{[0,\infty]}(O)(I):=W(O)(I) \wedge [0,\infty]_+.\] In terms of trees, the $[0,\infty]$ coordinate is the length of the root edge. This is an operad by grafting trees via the labeled root edge.
\end{dfn}

We will pay particular attention to the geometry of the $W_{[0,\infty]}$-construction; ultimately it is the key to defining the cooperad structure on the Koszul dual cooperad $B(-)$, and it will also be the key to encoding Koszul duality into parametrized spectra.

\begin{dfn}[The $W$-construction of a right module]
Given a right module $R$ in $(\mathrm{Top}_*,\wedge)$ over the operad $O$, let $T(R)(I)$ denote the space
of rooted trees such that every internal vertex has at least 2 children with labels as follows: the root $r$ is labeled by $R(e(r))$,  the internal vertices $v$ are labeled by $O(e(v))$, while nonleaf adjacent edges are
labeled by an element of $[0, \infty]$; we identify any tree containing a vertex labeled by a basepoint to a single point. We let $W(R)(I)$ denotes the quotient of $T(R)(I)$ by the relation that any length 0 edge can be collapsed
by applying right module or operad partial composition.
The collection of all $W(R)(I)$ is called $W(R)$. It is a right module over $W(O)$ by grafting via length $\infty$ edges, and it is a right module over $W_{[0,\infty]}(O)$ by grafting via the $[0,\infty]$ coordinate.
\end{dfn}

If $O$ is an operad in $(\mathrm{Top},\times)$ and $R$ is a right module over it, $W(O)$ and $W(R)$, or any of their variants, are defined by $W(O)(I):=W(O_+) - *$, $W(R)(I):=W(R_+)(I)-*$. In either the pointed or unpointed case, there are equivalences of operads and compatible equivalences of right modules induced by composition:
\[W(O) \rightarrow O,\]
\[W(R)\rightarrow R.\]
Similarly for the other versions of the $W$-construction.

There are many similarities in the definition of the operad $\mathcal{F}_n$ and the $W$-construction of an operad. There is a quite remarkable fact that $W(\mathcal{F}_n) \cong \mathcal{F}_n$ as operads \cite{salvatore_2021}. One could point to this as the primary reason to expect a Koszul self duality result for the operad $E_n \simeq \mathcal{F}_n$. We observed in \cite[Section 7]{malin2} that a minor variation of the proof $W(\mathcal{F}_n) \cong \mathcal{F}_n$ shows that $W(\mathcal{F}_M)\cong \mathcal{F}_M$ as right modules. In particular, these $W$-constructions turn out to be manifolds. This observation will end up simplifying arguments in Lemma \ref{lem:codimension0}.

\begin{dfn}[The boundary and interiors of the $W$-construction]
Let $O$ be an operad in $(\mathrm{Top},\times)$ and $\partial W(O)$ denote the subsymmetric sequence of $W(O)$ of trees with a length $\infty$ edge. We let $\mathring{W}(O):= W(O) -\partial W(O)$.
We define  \[\mathring{W}_{(0,\infty)}(O):=\mathring{W}(O) \times (0,\infty).\] It is an operad via grafting trees by the $(0,\infty)$ coordinate. 
Let \[\partial W_{[0,\infty]}(O):= (\partial W(O) \times [0,\infty]) \cup (W(O) \times \{0\} )\cup (W(O) \times \{\infty\}),\] i.e. the trees in $W_{[0,\infty]}(O)$ which have an edge of length $\infty$ or the root edge has length $0$.

\end{dfn}
\begin{dfn}
For an operad $O$ and right module $R$ in $(\mathrm{Top},\times)$, let  $\partial W(R)$ denote the subsymmetric sequence of $W(R)$ where any edge is length $\infty$. Let $\mathring{W}(R)=W(R)-\partial W (R)$. This is a right module over  $\mathring{W}_{(0,\infty)}(O)$ by grafting via the root edge.
\end{dfn}

%  The following operad equivalence is due to Salvatore, and the module version of the statement is analogous \cite{salvatore_2001}.

% \begin{lem}\label{lem:salvatore}
% There is an equivalence of operads $W(E_n) \rightarrow \mathcal{F}_n$ and, for a framed manifold $M$, compatible equivalences \[W(E_M)\rightarrow \mathcal{F}_M,\] \[W(E_{M^+}) \rightarrow \mathcal{F}_{M^+}.\] The latter are natural (covariantly, contravariantly resp.) with respect to open inclusion.
% \end{lem}

% The construction of the equivalences is not difficult, on the complement of $\partial W(-)$ one uses the edge lengths to scale down the various disks, then we apply operadic composition and record the centers of the disks. The formal construction of $\mathcal{F}_n,\mathcal{F}_M$ as completions of configuration spaces then imply this extends to the entire W-construction.

\begin{dfn}
A right module pair $(R,A)$ over an operad $O$ in $(\mathrm{Top},\times)$ is a right module $R$ with a chosen right submodule $A$ for which $A(I) \rightarrow R(I)$ is a cofibration for all finite sets $I$. The quotient $R/A$ is the right $O_+$-module in $(\mathrm{Top}_*,\wedge)$ given by $R(I)/A(I)$.
\end{dfn}

At this point, we observe that we may extend the construction of the right module $\mathcal{F}_M$ to compact framed manifolds with boundary. Fix a framed embedding $M \rightarrow N$ such that $N$ has no boundary, and let $\mathcal{F}_M$ be the right submodule of $\mathcal{F}_N$ of infinitesimal configurations contained in $M$. Based on the description of $\mathcal{F}_N$ in terms of infinitesimal configurations, it is clear that the isomorphism type of $\mathcal{F}_M$ depends only on the framed diffeomorphism type of $M$.

\begin{dfn}
If $M$ is a framed manifold with boundary, then 
$\widehat{\partial}\mathcal{F}_M$ is the right submodule of $\mathcal{F}_M$ of configurations which have a point in $\partial M$.
\end{dfn}
We use the notation $\widehat\partial$ rather than $\partial$ because $\mathcal{F}_M(I)$ is a manifold and $\partial \mathcal{F}_M(I)$ is larger than $\widehat\partial \mathcal{F}_M(I)$, in particular it also contains any configurations which have an infinitesimal component.

\begin{dfn}If $(R,A)$ is a right module pair, then $DW(R,A)$ is $W(A) \cup\partial W(R)  $.
\end{dfn}

\begin{dfn}
    An operad $O$ in $(\mathrm{Top},\times)$ is closed if $O$ has the levelwise nonequivariant homotopy type of a compact CW complex, and the image of every partial composite is closed. Similarly, a right module pair $(R,A)$ over $O$ is closed if $R$ and $A$ have the levelwise nonequivariant homotopy types of compact CW complexes, and the image of any of the restricted partial composites \[\mathrm{closure}(R(I)-A(I)) \times O(J) \rightarrow R(I \cup_a J)\] is closed.
\end{dfn}

The most relevant example of a closed operad is $\mathcal{F}_n$ which has 
 a closed right module pair $(\mathcal{F}_M,\hat\partial\mathcal{F}_M)$. Note that every operad which satisfies the required homotopy finiteness conditions is equivalent to a closed operad since $W(O)$ is easily seen to be closed. Similarly, if the right module pair $(R,A)$ satisfies the homotopy finiteness conditions, then $(W(R),W(A))$ is closed.

\begin{lem}\label{lem:operadinpair}
For a closed operad $O$ the pair $(W_{[0,\infty]}(O),\partial W_{[0,\infty]}(O))$ forms an operad 
 in $(\mathrm{Top}_\subset,\times)$. For a closed right module pair $(R,A)$ over $O$ the pair $(W(R),DW(R,A))$ forms a right module over $(W_{[0,\infty]}(O),\partial W_{[0,\infty]}(O))$.

\end{lem}

\begin{proof}
First, note that the quotient map condition is satisfied by the definition of the $W_{[0,\infty]}$-constructions. Now observe that that for the composite associated to $I \cup_a J$
 \[\partial W(O)_{[0,\infty]}(I \cup_a J) \cap \mathrm{image}(W(O)_{[0,\infty]}(I) \times W(O)_{[0,\infty]}(J))\] 
consists of trees which are the grafting of an $I$-labeled tree and $J$-labeled tree, which are necessarily unique, and either some edge has length $\infty$ or the root edge has length $0$. In other words the intersection is contained in the image of \[(W(O)_{[0,\infty]}(I) \times \partial W_{[0,\infty]}(O)(J)) \cup (\partial W_{[0,\infty]}(O)(I) \times W_{[0,\infty]}(O)(J)).\]
 
\noindent This implies the boundary condition for morphisms in $\mathrm{Top}_\subset$ is satisfied by the partial composites. Similarly, \[DW(R,A)(I \cup_a J) \cap \mathrm{image}(W(R)(I) \times W(O)_{[0,\infty]}(J))\] consists of the points which are the grafting of an $I$-labeled tree and a $J$-labeled tree, which are necessarily unique, and some edge has length $\infty$ or the root is labeled by $A$. In other words, the intersection is contained in the image of \[(DW(R,A)(I) \times W_{[0,\infty]}(O)(J)) \cup (W(R)(I) \times \partial W_{[0,\infty]}(O)(J))\] which show the right module composites also satisfy the boundary conditions. The second condition follows from the fact the partial composites of $\mathring{W}_{(0,\infty)}(O)$ and $\mathring{W}(R)$ are open inclusions \cite[Proposition 4.8, Proposition 6.8]{malin2}, and the third condition on closures follows from the behavior of the thickened $W$-construction with respect to length $0$ edges together with our requirement that $O$ and $(R,A)$ are closed.
\end{proof}

Let us now define the Koszul dual cooperad. This version appears in \cite[Section 6]{ching_salvatore} where it is observed that it coincides with the version originally defined in \cite{ching_2005}. It bears the name ``bar construction'' because Ching observed in his thesis that it is isomorphic to $B(1,O,1)$, a bar construction with respect to the $\circ$-product.

\begin{dfn}[The bar construction $B$]
Given an operad $O$ in $(\mathrm{Top}_*,\wedge)$, let $B(O)$ denote $W(O)\wedge
(0, \infty)^+$, where the $(0, \infty)^+$ coordinate is interpreted as a labeling of the edge adjacent to the root, modulo the relations that any tree with an $\infty$ length edge is identified with
the basepoint.

This is a cooperad via the decomposition which to an $I \cup_a J$ labeled tree $T$ returns, if possible, the unique $I$-labeled weighted tree and $J$-labeled weighted tree which graft along $a$ to obtain $T$. If this is not possible, we send it to $*$.

\end{dfn}

\begin{dfn}[The Koszul dual $K(-)$]
If $O$ is an operad in $(\mathrm{Top}_*,\wedge)$, the Koszul dual operad in $(\mathrm{Sp},\wedge)$ is \[K(O):= (\Sigma^\infty B(O))^\vee.\]
\end{dfn}

As a matter of taste, we will often write $K(\Sigma^\infty O)$ rather than $K(O)$ or $K(\Sigma^\infty_+ O)$ if $O$ is unpointed. This is justified by the fact that the bar construction naturally extends to reduced operads in spectra, and there is an isomorphism $B(\Sigma^\infty O)\cong \Sigma^\infty B(O)$. By Definition \ref{dfn:lax}, Lemma \ref{lem:operadinpair}, and Proposition \ref{prp:functor}, we may endow the sequence of Spivak normal fibrations of the pair $(O,\mathrm{indecom}^h(O))$ with the structure of an operad in parametrized spectra:

\begin{dfn}
For a closed operad $O$ in $(\mathrm{Top},\times)$ the Koszul dualizing fibration $\xi_O$ is the operad in $(\mathrm{ParSp},\bar \wedge)$
\[\xi_O(I)=P(W_{[0,\infty]}(O)(I),\partial W_{[0,\infty]}(O)(I)).\]
\end{dfn}

For an operad $O$ in parametrized spectra, we defined the Thom complex operad $\mathrm{Th}(O)$ in $(\mathrm{Sp},\wedge)$ (Definition \ref{dfn:thomOp}).

\begin{prp}[Koszul-Verdier duality] \label{prp:koszulverdieroperad}
For a  closed operad $O$ in $(\mathrm{Top},\times)$, there is an operad equivalence
\[ \mathrm{Th} (\xi_O) \xrightarrow{\simeq}K(\Sigma_+^\infty O). \]
\end{prp}

\begin{proof}
Observe that $\Sigma^\infty B(O_+)^\vee$ can be computed by applying the functor
\[(X,A) \mapsto X/A^\vee\]
to $(W_{[0,\infty]}(O),\partial W_{[0,\infty]}(O))$. By applying Proposition \ref{prp:functor}, we just need to check the following diagram commutes:
\begin{center}
 % https://tikzcd.yichuanshen.de/#N4Igdg9gJgpgziAXAbVABwnAlgFyxMJZABgBoBGAXVJADcBDAGwFcYkQAdDgW3pwAsATt2AAVfgF8AFAAUpADVIBBAJQqABOq4B3GFADmMLTz5CR46XICapAEJqQE0uky58hFGWLU6TVuy5eAWExSVkpBWUVLjxueCkbewcnF2w8AiJyCh8GFjZEEHkAeiV1AD0uWhgjHT1DdSsi2wqOKrYUkAw090zSbxpc-wKFEuNdAyNG+xa2xx86+CJQADNBCG4kMhAcCCQAJg7V9c2aHaRyQ7WNxD3T3cQAZkvjxCzt+6fKCSA
\begin{tikzcd}
{\mathrm{Th}(P((X,A)\times(Y,B)))} \arrow[r]                          & (\Sigma^\infty (X/A \wedge Y/B))^\vee               \\
{\mathrm{Th}(P(X,A))  \wedge \mathrm{Th}(P(Y,B))} \arrow[r] \arrow[u] & (\Sigma^\infty X/A) ^\vee \wedge (\Sigma^\infty Y/B)^\vee \arrow[u]
\end{tikzcd}
\end{center}
The left hand  map is defined by taking smash products of the fibers of $P(X,A)$ and $P(Y,B)$, which consist of relative sections of $\Sigma^\infty_X \mathrm{Path}(x,-)$ and $\Sigma^\infty_Y\mathrm{Path}(y,-) $, while the right hand map is defined by taking smash products of maps into the sphere spectrum $S^0$. On each fiber, passing from the left side to the right side is given by composition with the fiberwise collapse
\[\Sigma^\infty_X \mathrm{Path}(x,-) \rightarrow \Sigma^\infty_X X,\] and this commutes with taking smash products of sections.

\end{proof}

 In \cite{malin2}, we studied a homological precursor to topological Koszul self duality. Suppose $O$ is an operad in $(\mathrm{Top},\times)$ for which all spaces have the homotopy type of finite CW-complexes. Recall the definition of a Poincaré duality pair \cite{spivak_1967}.

\begin{dfn}
A pair $(X,A)$ is an $n$-dimensional Poincaré duality pair if there is an $\alpha \in H_n(X,A)$ so that 
\[H^*(X)\xrightarrow{\cap \alpha} H_{n-*}(X,A)\]
\[H^*(X,A)\xrightarrow{\cap \alpha} H_{n-*}(X)\]
\[H^*(A) \xrightarrow{\cap \partial \alpha} H_{n-*+1}(A)\]
are all isomorphisms. Here $\partial$ denotes the connecting homomorphism in homology for the pair $(X,A)$.
\end{dfn}
.
 \begin{dfn}
A distinguished $n$-class of an operad $O$ is a homology class $\alpha_I$, for each nonempty finite set $I$, in the relative homology \[H_{n|I|-n}(W_{[0,\infty]}(O)(I),\partial W_{[0,\infty]}(O)(I))\cong \bar{H}_{n|I|-n}(B(O)(I))\] such that under the partial decomposites 
\[\alpha_{I \cup_a J} \rightarrow \alpha_I \otimes \alpha_J.\]

\end{dfn}
\noindent As a sanity check, note that \[(n|I|-n) +( n|J|-n)=n(|I|+|J|-1)-n=n|I \cup_a J|-n.\]

\begin{dfn}
An operad $O$ with a distinguished $n$-class $\alpha$ is Poincaré-Koszul of dimension $n$ if $\alpha_I$ makes \[(W_{[0,\infty]}(O)(I),\partial W_{[0,\infty]}(O)(I))\] into a Poincaré duality pair for all $I$. We call $\alpha$ as the fundamental class.
\end{dfn}

 For the rest of our discussion on operads we fix a field $k$ and all (co)homology and tensor products are taken with respect to $k$. There is a natural notion of suspension for operads in the category $(\mathrm{dgVect}_k,\otimes)$. For a chain complex $A$, let $A[i]$ denotes the graded vector space where the gradings have been shifted up $i$.

\begin{dfn}
The algebraic $n$-sphere operad $S_n$ in $(\mathrm{dgVect}_k,\otimes)$ is defined by \[S_n(I):=k[n|I|-n]\] with partial composites determined by the canonical isomorphism $k \otimes k \cong k$. The algebraic $n$-sphere operad is equivalently defined as the coendomorphism operad \[\mathrm{CoEnd}(k[n]):= \mathrm{Map}(k[n],k[n]^{\otimes (-)}).\]
\end{dfn}

\noindent We define the $n$th suspension of an operad $O$ in $(\mathrm{dgVect},\otimes)$ by \[s_n O(I):= S_n(I) \otimes O(I).\] This naturally forms an operad. By unfolding definitions, one immediately obtains \cite[Theorem 5.5]{malin2}:

\begin{thm}
If $O$ is a Poincaré-Koszul operad of dimension $n$ with fundamental class $\alpha$, then there is an isomorphism of operads

\[H_*(O) \cong s_n H_*(K(\Sigma^\infty_+ O)).\]
induced by \[H^*(W_{[0,\infty]}(O)) \xrightarrow{\cap \alpha} s_{-n}\bar{H}_*(B(O)),\] where we consider $H^i(-)$ as living in degree $-i$. This isomorphism is natural with respect to maps of Poincaré-Koszul operads of dimension $n$ that preserve the fundamental $n$-class

\end{thm}
 
 Classically, there is a hierarchy of self duality:
 \begin{enumerate}
     \item Twisted Poincaré complex or equivalently the Spivak normal fibration is spherical,
     \item $\mathbb{Z}$-Poincaré complex or equivalently the Spivak normal fibration is also oriented 
     \item $\Sigma^\infty_+ X \simeq \Sigma^n \Sigma^\infty_+ X^\vee$ or equivalently the Spivak normal fibration is trivial \cite[Corollary 3.4]{wall_1967}.
 \end{enumerate}

\begin{prp}\label{prp:spherical}
A closed operad $O$ is Poincaré-Koszul with respect to all fields $k$, if and only if the fibers of $\xi_O$ are all spherical, each spherical fibration is $k$-orientable, and there are fixed orientations of the fibers such that the partial composites induce degree $1$ maps of the fibers.
\end{prp}

\begin{proof}
    The forwards direction follows from the well developed theory of Spivak normal fibrations \cite{klein_2007,klein_2001,wall_1967}, in particular the Spivak normal fibration of a Poincaré duality pair is spherical and respects pushout products. The backwards direction follows from Proposition \ref{prp:Poincaréduality}.
\end{proof}

In \cite[Theorem 5.10]{malin2}, we proved that the $E_n$ operad is Poincaré-Koszul, hence $\xi_{E_n}$ consists of spherical fibrations with orientations compatible with the operad partial composites. In analogy with the classical story, the Koszul self duality of $E_n$ should be equivalent to a structured trivialization of $\xi_O$. To formulate this, we must introduce $n$-sphere operads.

\begin{dfn} 
An $n$-sphere operad $S_n$ in $(\mathrm{Sp},\wedge)$ is an operad weakly equivalent to the coendomorphism operad given by $\mathrm{CoEnd}(S^n)$
\[\mathrm{CoEnd}(S^n)(I):=F(S^n,(S^n)^{\wedge I}).\]
\end{dfn}
 If $f:S^n \rightarrow (S^n)^{\wedge I},g:S^n \rightarrow (S^n)^{\wedge J}$, we may form the infiniteseimal composite
\begin{center}
% https://tikzcd.yichuanshen.de/#N4Igdg9gJgpgziAXAbVABwnAlgFyxMJZABgBoBGAXVJADcBDAGwFcYkQBlAPTAAIAdfgHcYUAOYwB-KBBxwpI8ZO59BiiVJlyFojSpABfUuky58hFGWLU6TVu31GT2PASJkATDYYs2iTjw6Spqy8mq6yoHhwYJaYcIRvCpBGrGhKZGEBjYRCCigAGYAThAAtkjkNDgQSGS2vuwFhsYgxWW1VTWIHjQ+9v7kGSHa0RpiUgDGWEUTvAVDaSMJweSGlAZAA
\begin{tikzcd}
S^n \arrow[d, "f"]                                                                                                     \\
S^n \wedge \dots \wedge S^n \wedge \dots \wedge S^n \arrow[d, "1 \wedge \dots \wedge g \wedge \dots \wedge 1"] \\
S^n \wedge \dots \wedge S^n \wedge \dots \wedge S^n \wedge \dots \wedge S^n                                           
\end{tikzcd}
 \end{center}

This construction gives rise to the partial composites of the coendomorphism operad. The only aspect of $S_n$ we need is that $H_*(S_n)$ is the algebraic $n$-sphere operad. We will abuse notation by defining operadic suspension without an implicit model of $S_n$ in mind. 

\begin{dfn}\label{dfn:opSus}
For an operad $P$ in $(\mathrm{Sp},\wedge)$, $s_n P$ is $ S_n \wedge P.$

\end{dfn}

\begin{dfn}
An operad $O$ in $(\mathrm{Top},\times)$ is Koszul self dual of dimension $n$ if there is a zigzag equivalence from $s_n K(\Sigma^\infty_+ O)$ to $\Sigma^\infty_+ O$. We call such a zigzag equivalence a Koszul zigzag.
\end{dfn}

Of course, because there are model category structures on $\mathrm{Operad}(\mathrm{Sp},\wedge)$ with weak equivalences given by the levelwise weak equivalences, the existence of a Koszul zigzag is equivalent to $\Sigma^\infty_+ O \cong s_n K(\Sigma^\infty_+ O)$ in the homotopy category of operads \cite{kro_2007}. 

Recall that the fiberwise smash product of a parametrized spectrum $p$ with a spectrum $E$ is defined as $p \bar\wedge E$, where we consider $E$ as a parametrized spectrum over $*$. We can similarly define fiberwise suspension of operads in parametrized spectra:

\begin{dfn}
    The $n$-fold fiberwise suspension $P \bar\wedge S_n$ of an operad $P$ in $(\mathrm{ParSp},\bar\wedge)$ is 
    \[(P \bar\wedge S_n)(I)=P(I) \bar\wedge S_n(I)\]
\end{dfn}
For convenience, given a spectrum $E$ we will abbreviate $\Sigma^\infty_X X \bar\wedge E$ by $X \bar\wedge E$. Recall that because all of our parametrized operads $O$ lie in $\mathrm{ResOp}$ (Definition \ref{dfn:resOp}), the zero sections of $O$ form an operad in $(\mathrm{Top},\times)$ called $\mathrm{Base}(O)$ which is preserved by parametrized operad maps. We say $O$ covers $\mathrm{Base}(O)$.

\begin{dfn}
For an operad $P$ in parametrized spectra, covering an operad $O$ in $(\mathrm{Top},\times)$ a zigzag equivalence from $P$ to $O \bar\wedge S_n$ is called a (spherical) $n$-trivialization.
\end{dfn}

\begin{thm}
A closed operad $O$ in $(\mathrm{Top},\times)$ is Koszul self dual of dimension $n$, if and only if there is a $(-n)$-trivialization of $\xi_O$.
\end{thm}

\begin{proof}
The backwards direction is obvious given Koszul-Verdier duality. The forward direction is not difficult, but somewhat surprising. First, note that there is a zigzag map of operads, from $\operatorname{Th}(\xi_O)$ to $ S_{-n}$ which comes from the Koszul self duality of $O$ and the map $O\rightarrow \mathrm{com}$, and this latter map, of course, induces isomorphism on $H_0$ when restricted to the individual path components of each $O(I)$. We conclude that there is a zigzag map $\mathrm{Th}(\xi_O) \rightarrow S_{-n}$ which, when restricted to the wedge summands corresponding to different path components of the base of $\xi_O$, induces isomorphisms on bottom homology. At this point, note that if this was a direct map instead of a zigzag, we could use the identification of maps out of Thom complexes to directly construct a map of operads $\xi_O \rightarrow W_{[0,\infty]}(O) \bar\wedge S_{-n}$. We could then argue that this is a fiberwise equivalence as follows:

By construction, this map is a fiberwise equivalence, if and only if the composite $\xi_O \rightarrow W_{[0,\infty]}(O) \bar\wedge S_{-n} \rightarrow S_{-n}$ is a fiberwise equivalence. It suffices to show this for each path component of $W_{[0,\infty]}(O)(I)$. Since the fibers of $\xi_O$ are spherical for a Koszul self dual operad by Proposition \ref{prp:spherical}, it then suffices to show these are fiberwise homology equivalences. This, in turn, would be implied by the map $\xi_O(I) \rightarrow  S_{-n}(I)$ yielding an isomorphism on the bottom homology of the Thom complexes when restricted to the wedge summands corresponding to the various path components. This is because by \cite[Lemma 3.1]{wall_1967} the bottom homology of these wedge summands is generated by the inclusion of a fiber. However, we already saw that after restricting to wedge summands, this map induces an isomorphism on bottom homology, and so we would be done.

In fact, the same idea works for the zigzag. Start with the map $\mathrm{Th}(\xi_O) \xrightarrow{\simeq} K(\Sigma^\infty_+ O)$; using the identification of maps out of a Thom complex, we have a map \[\xi_O \rightarrow W_{[0,\infty]}(O) \bar\wedge K(\Sigma^\infty_+ O).\] Note the base of either operad is $W_{[0,\infty]}(O)$ and this is, of course, not an equivalence as the fibers of one are spherical and the fibers of the other are Koszul duals of $O$. Applying $W_{[0,\infty]}(O) \bar\wedge -$ to the zigzag map from $K(\Sigma^\infty_+ O)$ to $S_{-n}$ results in a zigzag equivalence of operads starting at $\xi_O$ and ending at $W_{[0,\infty]}(O) \bar\wedge S_{-n}$ because if we invert all the backwards equivalences (in the homotopy category of parametrized spectra) the composition
\[\xi_O \rightarrow W_{[0,\infty]}(O) \bar\wedge S_{-n}\]
is an equivalence by the previous argument.

\end{proof}

In this proof, we see the distinction of zigzag equivalences and equivalences of zigzags come into play. Even if we started with a Koszul zigzag for which every map was a weak equivalence, the trivialization we produce only has the weaker property that after inverting all backwards maps and composing it becomes an equivalence. 

 In fact, we actually proved something stronger.

\begin{thm}\label{thm:strongKoszul}
For a closed operad $O$ in $(\mathrm{Top},\times)$ the following are equivalent:

\begin{enumerate}
  \item $O$ is Koszul self dual of dimension $n$.
  \item There is a $(-n)$-trivialization of $\xi_O$.
  \item $O$ is Poincaré-Koszul, and there is a zigzag map of operads from $K(\Sigma^\infty_+ O) $ to $ S_{-n}$ such that the map 
  \[H_{n-n|I|}(K(\Sigma^\infty_+ O)(I)) \rightarrow H_{n-n|I|}(S_{-n}(I))\]
  restricts to an isomorphism on the wedge summands of $K(\Sigma^\infty_+O)(I)$ corresponding to the different path components of $O(I)$.
\end{enumerate}

\end{thm}

\begin{proof}
We have already proven the equivalence of $(1),(2)$, and the implication $(1) \implies (3)$ was also shown above. Now we show $(3) \implies (2)$. Assume we are given such a zigzag map of operads from $K(\Sigma^\infty_+ O)$ to $S_{-n}$, we can construct a zigzag map from $\xi_O$ to $W_{[0,\infty]}(O) \bar\wedge S_{-n}$ in the same way as in the above proof. Poincaré-Koszulness detects that $\xi_O$ is spherical by Proposition \ref{prp:spherical}, hence the bottom homology of a wedge summand of the Thom complex corresponding to a path component is generated by the inclusion of a fiber. Our assumption then implies that the map is a fiberwise equivalence.
\end{proof}

\begin{remark}
Knudsen has constructed the ``universal enveloping algebra'' functor \[\mathrm{Alg}_{\mathrm{bar}(\mathrm{com})^\vee}(\mathrm{Sp},\wedge) \rightarrow \mathrm{Alg}_{E_n}(\mathrm{Sp},\wedge)\] where $\mathrm{bar}$ denotes Lurie's bar construction from operads to cooperads \cite{knudsen_2018}. It seems likely this arises from induction along a map $s_n \mathrm{bar}(\mathrm{com})^\vee \rightarrow \Sigma^\infty_+ E_n$. If this were the case, applying $\mathrm{bar}(-)^\vee$ again, would construct a map $\mathrm{bar}(\Sigma^\infty_+ E_n)^\vee \rightarrow S_{-n}$. If there is a similar Thom complex description for $\mathrm{bar}(-)^\vee$, the same argument as above would give an equivalence of operads $\Sigma^\infty_+ E_n \simeq s_n \mathrm{bar}(\Sigma^\infty_+ E_n)^\vee$.
\end{remark}

 We now continue with theory of Koszul duality for right modules, all proofs are almost identical to the operad case, so we do not include them.

\begin{dfn}
If $R$ is a right module over $O$ in $(\mathrm{Top}_*,\wedge)$, let $B(R)$ denote $W(R)$  modulo the relations that any tree with a length $\infty$ edge is identified with
$*$.
This is a right comodule via decomposing trees, if possible, and otherwise sending collapsing to the basepoint.
If we wish to specify the operad we are taking bar construction with respect to, we write it as $B(R,O,1)$.

\end{dfn}

\begin{dfn}
If $R$ is a right module over the operad $O$ in $(\mathrm{Top}_*,\wedge)$, the Koszul dual right module is \[K(R):= (\Sigma^\infty (B(R)))^\vee.\]
\end{dfn}

Again, as a matter of taste, we will write $B(\Sigma^\infty R)^\vee$ or $B(\Sigma^\infty_+ R)$ if $R$ is unpointed, knowing that this abuse of notation is justified by the extension of Koszul duality to operads and right modules in spectra. Recall that all of our operads and right modules are levelwise nonequivariantly homotopy finite.

\begin{dfn}
For a closed right module pair $(R,A)$ over a closed operad $O$ in $(\mathrm{Top},\times)$ the Koszul dualizing fibration $\xi_{(R,A)}$ is the right module over $\xi_O$ in $(\mathrm{ParSp},\bar\wedge)$ 
\[\xi_{(R,A)}(I):=P(W(R),DW(R,A)).\]
\end{dfn}

\begin{prp}[Koszul-Verdier duality for right modules]\label{prp:koszulverdiermodule}
For a closed right module pair $(R,A)$ over a closed operad $O$ in $(\mathrm{Top},\times)$, there is a Koszul-Verdier duality equivalence of right modules compatible with the Koszul-Verdier duality equivalence of $O$,
\[\operatorname{Th}(\xi_{(R,A)}) \xrightarrow{\simeq}  K(\Sigma^\infty R/A). \]
\end{prp}

As in the operad case, we can describe a class of right module pairs which has an automatic relative homological Koszul self duality map. We recall a slight generalization of the definitions and results from \cite{malin2}. We fix an operad $O$ in $(\mathrm{Top},\times)$ and right module pair $(R,A)$ all of which are levelwise nonequivariantly homotopy finite. For the rest of our discussion on right modules we fix a field $k$ and all (co)homology and tensor products are taken with respect to $k$.

\begin{dfn}
A distinguished $(n,d)$-class of a right module pair $(R,A)$ over an operad $O$ with a distinguished $n$-class $\alpha$ is a choice for each nonempty finite set $I$ of an element $\beta_I$ in \[H_{n|I|-n+d}(W(R)(I),DW(R,A)(I))\cong \bar{H}_{n|I|-n+d}(B(R/A)(I)),\] such that under the partial decomposites
\[\beta_{I \cup_a J} \rightarrow \beta_I \otimes \alpha_J.\]
\end{dfn}
\noindent As a sanity check, note \[(n|I|-n+d) + (n|J|-n)=n(|I|+|J|-1)-n+d=n|I \cup_a J|-n+d.\]

\begin{dfn}
A right module pair $(R,A)$ with a distinguished class $\beta$ over a Poincaré-Koszul operad $O$ is Poincaré-Koszul of dimension $(n,d)$ if $\beta_I$ makes \[(W(R)(I),DW(R,A))\] into a Poincaré duality pair for all $I$.
We call $\beta$ the fundamental class.
\end{dfn}

\begin{dfn}
The algebraic $(n,d)$-sphere right module $S_{(n,d)}$ over the algebraic $n$-sphere operad $S_n$ is defined by \[S_{(n,d)}(I):=k[n|I|-n+d]\] with partial composites determined by the canonical isomorphism $k \otimes k \cong k$. It is equivalently defined as the levelwise $d$-fold suspension $ S_n[d]$.
\end{dfn}
 We define the $(n,d)$-suspension of a right $P$-module $Q$ in $(\mathrm{dgVect}_k,\otimes)$ by \[s_{(n,d)} Q(I):= S_{(n,d)}(I) \otimes Q(I).\] This naturally forms a right module over $s_n P$. As such, the suspension of right modules is a combination of two natural notions of suspension: one which is internal to right $O$-modules which we write as $\Sigma^d$, and a more interesting suspension which transforms right $O$-modules to right $s_n O$-modules. By unraveling definitions one immediately obtains \cite[Theorem 7.4]{malin2}.

\begin{thm}
If $(R,A)$ is a Poincaré-Koszul right module pair of dimension $(n,d)$ with fundamental class $\beta$ there is an isomorphism of right modules

\[H_*(R) \cong s_{(n,d)} \bar{H}_*(K(\Sigma^\infty R/A)).\]
induced by $H^*(R) \xrightarrow{\cap \beta} s_{(-n,-d)}H_*(B(R/A))$, where we use the convention that $H^i(-)$ lives in degree $-i$. This is compatible with the Poincaré-Koszul duality isomorphism of $O$. This isomorphism is natural with respect to maps of Poincaré-Koszul right module pairs that preserve the fundamental $(n,d)$-class.

\end{thm}

Just as in the operad case we have a topological characterization of Poincaré-Koszul right module pairs.

\begin{prp}
A closed right module pair $(R,A)$ over a  Poincaré-Koszul operad $O$ is Poincaré-Koszul with respect to all fields $k$, if and only if the fibers of $\xi_{(R,A)}$ are all spherical, each spherical fibration is $k$-orientable, and there are fixed orientations of the fibers such that the partial composites induce degree $1$ maps of the fibers.
\end{prp}

In \cite[Theorem 7.8]{malin2}, we proved that for a compact, framed manifold with boundary $M$, the closed right module pair $(\mathcal{F}_M,\widehat\partial \mathcal{F}_M)$ was Poincaré-Koszul of dimension $(n,n)$. Hence all the fibers of $\xi_{(\mathcal{F}_M,\widehat\partial \mathcal{F}_M)}$ are spherical with compatible orientations. Ultimately, we will show that $\xi_{(\mathcal{F}_M,\widehat\partial \mathcal{F}_M)}$ is in fact trivial, and, as a consequence, that $(\mathcal{F}_M,\widehat\partial \mathcal{F}_M)$ is Koszul self dual. Observe that for a right module $R$ over an operad $O$ in $(\mathrm{Top}_*,\wedge)$ or $(\mathrm{Sp},\wedge)$, the levelwise suspension $(\Sigma^d R)(I):=\Sigma^d R(I)$ is still a right $O$-module.

\begin{dfn} 
An $(n,d)$-sphere right module $S_{(n,d)}$ over an $n$-sphere operad $S_n$ is a right module in $(\mathrm{Sp},\wedge)$ with a zigzag equivalence to $\Sigma^d \mathrm{CoEnd}(S^n)$ compatible with a zigzag equivalence from $S_n$ to $\mathrm{CoEnd}(S^n)$.
\end{dfn}

As before, we will abuse notation by not specifying a specific model of $S_{(n,d)}$ when defining right module suspension.

\begin{dfn}\label{dfn:modSus}
For a right module $Q$ over $P$ in $(\mathrm{Sp},\wedge)$ the $(n,d)$-suspension of $Q$ is the right $s_n P$-module,
\[s_{(n,d)} Q := S_{(n,d)} \wedge Q.\]

\end{dfn}

\begin{dfn}
A right module pair $(R,A)$ over an operad $O$ in $(\mathrm{Top},\times)$ is Koszul self dual of dimension $(n,d)$ with respect to a Koszul zigzag of $O$ if there is a zigzag equivalence, called a Koszul zigzag of $(R,A)$, from $s_{(n,d)}K(\Sigma^\infty R/A)$ to $\Sigma^\infty_+ R$, compatible with the Koszul zigzag of O.
\end{dfn}

\begin{dfn}
Assume we have a right module $Q$ over an operad $P$ in parametrized spectra which covers a right module $R$ over an operad $O$ in $(\mathrm{Top},\times)$. Given an $n$-trivialization of $P$, a zigzag equivalence from $Q$ to $ R \bar\wedge S_{(n,d)}$ which is compatible with the trivializaton of $P$ is called a (spherical) $(n,d)$-trivialization.
\end{dfn}

\begin{thm}\label{thm:moduleDuality}
For a closed operad $O$ in $(\mathrm{Top},\times)$ and a closed right module pair $(R,A)$ over $O$ the following are equivalent:

\begin{enumerate}
  \item $O$ is Koszul self dual of dimension $n$ and $(R,A)$ is Koszul self dual of dimension $(n,d)$.
  \item There is a $(-n)$-trivialization of $\xi_O$ and a $(-n,-d)$-trivialization of $\xi_{(R,A)}$.
  \item  $O$ and $(R,A)$ are Poincaré-Koszul of dimension $n,(n,d)$, respectively, and there is a pair of compatible zigzag maps of operads and right modules from $K (\Sigma^\infty_+ O) $ to  $S_{-n}$ and 
$K(\Sigma^\infty R/A)$ to $ S_{(-n,-d)}$ such that the maps
  \[H_{n-n|I|}(K(\Sigma^\infty_+ O)(I)) \rightarrow H_{n-n|I|}(S_{-n}(I))\]
  \[H_{n-n|I|-d}(K(\Sigma^\infty_+ R)(I)) \rightarrow H_{n-n|I|-d}(S_{(-n,-d)}(I))\]
  restrict to isomorphisms on the wedge summands of $K(\Sigma^\infty_+ O)(I)$ and $K(\Sigma^\infty_+ R)(I)$ corresponding to the different path components of $O(I)$ and $R(I)$.
\end{enumerate}

\end{thm}

\section{Self duality for submanifolds of $\mathbb{R}^n$} \label{sec:koszulself}
In this section we will show how the existence of a Koszul zigzag for $\Sigma^\infty_+\mathcal{F}_n$ gives rise to a Koszul zigzag for $(\mathcal{F}_M,\widehat\partial \mathcal{F}_M)$ when $M$ is a compact, codimension $0$ submanifold of $\mathbb{R}^n$. It is known that such Koszul zigzags exist for $\Sigma^\infty_+ \mathcal{F}_n$ \cite[Theorem 1.1]{ching_salvatore}:

\begin{thm}[Ching-Salvatore]
There is a zigzag of equivalences of operads from $\Sigma^\infty_+ \mathcal{F}_n$ to the $n$-fold suspension of its Koszul dual: \[\Sigma^\infty_+ \mathcal{F}_n \simeq \dots \simeq  s_n K(\Sigma^\infty_+ \mathcal{F}_n).\]
\end{thm}

\noindent For the rest of this section, we fix any such Koszul zigzag.

\begin{lem}\label{lem:contractible}
There is an equivalence of right $B((\mathcal{F}_n)_+)$-comodules \[B(\mathcal{F}_{(\mathbb{R}^n)^+}):=B((\mathcal{F}_{(\mathbb{R}^n)^+}),(\mathcal{F}_n)_+,1) \xrightarrow{\simeq} \Sigma^n B(1,(\mathcal{F}_n)_+,1)=:\Sigma^n B((\mathcal{F}_n)_+).\]
\end{lem}

\begin{proof}
 We will show there is an equivalence of right modules $\mathcal{F}_{(\mathbb{R}^n)^+} \rightarrow \Sigma^n 1$. Applying $B$ yields a map of right comodules which implies the result since taking $\Sigma^n$ commutes with bar constructions of right modules.
If $|I|=1$, $\mathcal{F}_{(\mathbb{R}^n)^+}(*)=S^n=  \Sigma^n 1(*)$, so we take the map to be the identity. In all other degrees the map is forced to be constant, so we must show $\mathcal{F}_{(\mathbb{R}^n)^+}(I)$ is contractible if $|I|>1$. By collar neighborhoods, we may assume we are working with the subspace with no infinitesimal configurations. In this case, we scale by $t \in [1,\infty]$ to contract to the basepoint. This is continuous since for all such configurations, there is at least one point not on the origin. 
\end{proof}

\begin{prp}\label{prp:loctrivial}
There is an $n$-trivialization of $\xi_{\mathcal{F}_n}$ and a compatible $(n,n)$-trivialization of $\xi_{(\mathcal{F}_{D^n},\widehat\partial\mathcal{F}_{D^n})}$.
\end{prp}
\begin{proof}
By Theorem \ref{thm:moduleDuality}, it suffices to show $(\mathcal{F}_{D^n},\widehat\partial\mathcal{F}_{D^n})$ is Koszul self dual of dimension $(n,n)$. There are straightforward equivalences of right modules \[\mathcal{F}_{D^n}  \simeq \mathcal{F}_{\mathbb{R}^n}  \simeq \mathcal{F}_n .\] Hence, a Koszul zigzag for $\mathcal{F}_n$ together with the observation $\mathcal{F}_{(\mathbb{R}^n)^+} \cong \mathcal{F}_{D^n}/\widehat\partial\mathcal{F}_{D^n}$ and the above lemma, implies there is a zigzag of right module equivalences from $s_{(n,n)}K(\mathcal{F}_{D^n}/\widehat\partial\mathcal{F}_{D^n})$ to $\Sigma^\infty_+ \mathcal{F}_{D^n}$ compatible with the Koszul zigzag from $s_n K(\Sigma^\infty_+\mathcal{F}_n)$ to $\Sigma^\infty_+\mathcal{F}_n$ after extending the Koszul zigzag by the identities of $\Sigma^\infty_+ \mathcal{F}_n$ and $s_{(n,n)}K(\Sigma^\infty_+ \mathcal{F}_n)$ to account for the application of the equivalences $\mathcal{F}_{D^n}  \simeq \mathcal{F}_{\mathbb{R}^n}  \simeq \mathcal{F}_n$.
\end{proof}

\begin{lem} \label{lem:codimension0}
If $M$ is a compact, codimension $0$ submanifold of $\mathbb{R}^n$, then $\xi_{(\mathcal{F}_M,\widehat\partial\mathcal{F}_M)}$ has an $(n,n)$-trivialization. Hence, there is a zigzag equivalence \[ s_{(n,n)}K(\mathcal{F}_{M}/\widehat\partial \mathcal{F}_M)) \quad  \mathrm{to}\quad \Sigma^\infty_+\mathcal{F}_M\] which can be taken to be natural with respect to inclusion. 
\end{lem}
\begin{proof}
Without loss of generality, we prove the result for manifolds $M$ embedded in $D^n$. As such, $\partial D^n \cap M \subset \partial M$. Note that the induced right module map \[(W(\mathcal{F}_M),DW(\mathcal{F}_M,\widehat\partial\mathcal{F}_M)) \rightarrow (W(\mathcal{F}_{D^n}),DW(\mathcal{F}_{D^n},\widehat\partial\mathcal{F}_{D^n}))\] lies in $\mathrm{Top}_\subset$ because it is true of the map \[(\mathcal{F}_M,\widehat\partial\mathcal{F}_M)\rightarrow (\mathcal{F}_{D^n},\widehat\partial\mathcal{F}_{D^n}).\] To produce the required trivialization, it then suffices to show the induced map
\[\xi_{(\mathcal{F}_M,\widehat\partial\mathcal{F}_M)} \rightarrow \xi_{(\mathcal{F}_{D^n},\widehat\partial\mathcal{F}_{D^n})}\] is an equivalence on fibers since by Proposition \ref{prp:loctrivial} the latter has a trivialization which can be pulled back to a trivialization of $\xi_{(\mathcal{F}_M,\widehat\partial\mathcal{F}_M)}$. By Corollary \ref{cor:spherical}, it suffices to show that for all finite sets $I$ the induced map 
%Made edit here to correct order
\[\bar{H}_{n|I|}(B(\mathcal{F}_{D^n}/\widehat\partial \mathcal{F}_{D^n})(I);k) \rightarrow \bar{H}_{n|I|}(B(\mathcal{F}_{M}/\widehat\partial \mathcal{F}_M)(I);k)  \] is an isomorphism for all choices of $k$ when restricted to the wedge summands corresponding to the different path components of $M$. One observes as in \cite{salvatore_2021}, that the $W$-construction can be interpreted as adding a collar to the manifold $\mathcal{F}_M(I)$, which implies the map of $W$-constructions is actually a codimension 0 embedding of topological $n|I|$-manifolds! Since the map on bar constructions is the associated collapse map, it induces isomorphisms on top degree homology when we restrict our attention to the individual path components.
\end{proof}

\section{Weiss cosheaves in the category $\mathrm{RMod}_{O}$} \label{sec:factorization}

In this section, we study Weiss cosheaves taking values in $\mathrm{RMod}_{O}, \Sigma\mathrm{Seq}(\mathrm{Sp})$, and we use them to prove the compactly supported Koszul self duality of all tame, framed $n$-manifolds. It is very often in manifold theory that showing a property $P$ holds for submanifolds of $\mathbb{R}^n$ implies that it actually holds for all $n$-manifolds. One way to approach such an argument is to demonstrate that the property $P$ can be deduced from a statement about homotopy (co)sheaves for a particular family of open covers, and then we use this to show the local result implies the global result.

There is a particularly relevant family of open covers called Weiss covers which are defined to be the open covers for which every finite subset of the manifold is contained inside some open of the cover. It is known that, as a symmetric sequence, the collection of configuration spaces is a homotopy cosheaf with respect to Weiss covers \cite[Lemma 2.5]{campos_ricardo_idrissi}. Using this fact, we show that the assignments
\[M \mapsto \Sigma^\infty_+ \mathcal{F}_M\]
\[M \mapsto s_{(n,n)}\Sigma^\infty K(\mathcal{F}_{M^+}) \]
are topological Weiss cosheaves with values in $\mathrm{RMod}_{\Sigma^\infty_+ \mathcal{F}_n},\mathrm{RMod}_{s_n K(\Sigma^\infty_+ \mathcal{F}_n)}$, respectively. Using a classification result of Ayala-Francis regarding locally constant Weiss cosheaves, we reduce the problem of finding a natural equivalence between these functors to finding a natural equivalence of their restrictions to the category of open subsets of $\mathbb{R}^n$. The result will then follow from the special case of codimension 0 submanifolds of $\mathbb{R}^n$, by the naturality of Lemma \ref{lem:codimension0}.
%combination of framed corollary 2.24 and loclaization

Many flavors of categorical homotopy theory appear in the section, related by the functors below, which we introduce when needed. We let superscripts of categories denote enrichments, superscripts ``bi'' represent passage to bifibrant objects, and use $\mathrm{QuasiCat}$ to denote the collection of quasicategories, sometimes simply called $\infty$-categories.

\begin{center}
% https://tikzcd.yichuanshen.de/#N4Igdg9gJgpgziAXAbVABwnAlgFyxMJZABgBpiBdUkANwEMAbAVxiRAB12BbOnACwBOXYAGFeAXwB6wTj35DgAFQhpx4kONLpMufIRQBGclVqMWbWb0HCxOKTO5WFAaTpg1GrSAzY8BIkYGJvTMrIgcjvLCALLQMAy29pZRwADKqTB26pravnpEAEzG1CHm4cnWwACKTHTYiZ65uv4oZEElZmERcpWxsAkS0hUKyqrZJjBQAObwRKAAZgIQXEhkIDgQSEamoRaRlalYYFPZXovLSEXrm4jbDHQARvEACjp++iAMMPM4IB275X2CgE8HEAB8APoOHoKB5YeZYB4CNxZRogc4rW7UDZIADM-zK3SsAGNGMAAHKnBZLTFXHGIfGfR4vN75cICLBTPi-AldCqkhgUpJA4RcOIMDw5dE0pAAFmxNyu9yeDFeeRaIA5XJ5O0Jw2Eh2OVOlF0Q8uuq15exhwhBcHBUP1wDhCKRKMlFHEQA
\begin{tikzcd}
\mathrm{Cat}^{\mathrm{Top}} \arrow[r, "\mathrm{Sing}"]                                                       & \mathrm{Cat}^{\mathrm{Kan}} \arrow[r, "\mathcal{N}"]                                                                          & \mathrm{QuasiCat} \\
\mathrm{ModelCat}^{\mathrm{Top}} \arrow[r, "\mathrm{Sing}"'] \arrow[u, "\mathrm{res}|_{\mathrm{bifibrant}}"] & \mathrm{ModelCat}^{\mathrm{SSet}} \arrow[u, "\mathrm{res}|_{\mathrm{bifibrant}}"] \arrow[ru, "\mathcal{N}^{\mathrm{model}}"'] &                  
\end{tikzcd}
\end{center}

We recall some definitions relevant to Weiss cosheaves and the quasicategories of manifolds studied in \cite{ayala_francis_2015}. An introduction to the more $\infty$-categorical aspects of Weiss cosheaves can be found in \cite{miller_ayala_francis_2020} and a reference for general use is \cite{brito_weiss_2013}. We make use of Definition \ref{dfn:framedemb}, which is the ``homotopically correct'' definition of framed embeddings.
%theorem 9.8 of chain rule gives simplicial + cofibrantly generated, using projective

%just need locally presentable

%that the nerve is symmetric monoidal follows from using injective model structure + lurie
%Maybe https://arxiv.org/pdf/2003.00071.pdf for enriched presheaf cat

% \begin{dfn}
%     The symmetric monoidal topological category $(\mathscr{M}\mathrm{an}_n,\sqcup)$ has objects the tame, smooth $n$-manifolds and morphism spaces $\mathrm{Emb}(M,N)$.

%     The symmetric monoidal topological category $(\mathscr{M}\mathrm{an}_n^{\mathrm{fr}},\sqcup)$ has objects the tame, smooth $n$-manifolds with a choice of framing and morphism spaces $\mathrm{Emb}^{\mathrm{fr}}(M,N)$.
% \end{dfn}

% \begin{dfn}
%     Given a map $B \rightarrow BO(n)$, the $\infty$-category of $B$-framed manifolds is the pullback
    
%     % https://tikzcd.yichuanshen.de/#N4Igdg9gJgpgziAXAbVABwnAlgFyxMJZABgBpiBdUkANwEMAbAVxiRAB12BbOnACwDGjYAFkAvpx78ATl2B0wYgHoAhAPqExpdJlz5CKAIzkqtRizaTefWcAAqENGLXAA9CrEgtO7HgJEyQ1N6ZlZEDm5rIQZRCUiZOQVnTW0QDF99ImMg6hCLcKsE+0dnNxUAeQAKMABKTzFTGCgAc3giUAAzaQguJAAmahwIJABmbxAunqRjECHR8cnexDJZ4cQ+he6llbnEQwaxIA
% \begin{center}
% \begin{tikzcd}
% \mathscr{M}\mathrm{an}^B_n \arrow[d] \arrow[r] & \mathrm{Top}_{/B} \arrow[d] \\
% \mathscr{M}\mathrm{an}_n \arrow[r]             & \mathrm{Top}_{/BO(n)}      
% \end{tikzcd}
% \end{center}
% where the bottom map is the tangent classifier \cite[Section 2.3]{miller_ayala_francis_2020}.
% \end{dfn}
\begin{dfn}
 The topological category $\mathscr{M}\mathrm{fld}_n^{\mathrm{fr}}$ has objects the tame, smooth $n$-manifolds with a choice of framing and the morphism space from $M$ to $N$ given by $\mathrm{Emb}^{\mathrm{fr}}(M,N)$.
\end{dfn}

Recall the singular set functor $\mathrm{Sing}:\mathrm{Top}\rightarrow \mathrm{SSet}$ is characterized as being adjoint to geometric realization. It is well known to takes values in Kan complexes and respect products, so it induces a functor $\mathrm{Sing}:\mathrm{Cat}^{\mathrm{Top}}\rightarrow \mathrm{Cat}^{\mathrm{Kan}}$. Let \[\mathcal{N}:\mathrm{Cat}^{\mathrm{Kan}} \rightarrow \mathrm{QuasiCat}\] denote the homotopy coherent nerve \cite[Section 1.1.5]{lurie_2009}.

\begin{dfn}
    The quasicategory $\mathscr{M}\mathrm{fld}_n^{*}$ is $\mathcal{N}(\mathrm{Sing}(\mathscr{M}\mathrm{fld}_n^{\mathrm{fr}}))$.
\end{dfn}

The notation of this quasicategory of manifolds is compatible with the notation of Ayala-Francis. By our careful choice of Definition \ref{dfn:framedemb}, $\mathcal{N}(\mathrm{Sing}(\mathscr{M}\mathrm{fld}_n^{\mathrm{fr}}))$ is a model of the category of $B$-framed manifolds $\mathscr{M}\mathrm{fld}_n^B$ when $B$ is a point \cite[Definition 2.17]{ayala_francis_2019}. Hence, the general theory of \cite{ayala_francis_2015} applies to the study of this quasicategory. We will  use vocabulary from \cite{ayala_francis_2015} regarding $\infty$-categories of manifolds, but this is confined to our short discussion of Weiss cosheaves. We will always be clear to distinguish the topological category of framed manifolds $\mathscr{M}\mathrm{fld}_n^\mathrm{fr}$ from the quasicategory of framed manifolds $\mathscr{M}\mathrm{fld}_n^\mathrm{*}$.

In order to construct well behaved functors out of the quasicategory $\mathscr{M}\mathrm{fld}_n^{*}=\mathcal{N}(\mathrm{Sing}(\mathscr{M}\mathrm{fld}_n^{*}))$ we work with the following paradigm: 

\begin{enumerate}
\item Use point-set topology to construct a continuous functor $F:\mathscr{M}\mathrm{fld}^{\mathrm{fr}}_n \rightarrow V$ of topological categories.
\item Take $\mathrm{Sing}(-)$ to arrive at a functor between Kan complex enriched categories.
\item Equip $\mathrm{Sing}(V)$ with the structure of a simplicial model category.
\item Postcompose $\mathrm{Sing}(F)$ with a simplicial bifibrant replacement functor to land in $\mathrm{Sing}(V)^\mathrm{bi}$.
\item Apply the homotopy coherent nerve $\mathcal{N}(-)$ to land in $\mathcal{N}(\mathrm{Sing}(V)^\mathrm{bi})$.
\end{enumerate}

There are two potential roadblocks: (1) the existence of a simplicial model structure and (2) the existence of an enriched bifibrant replacement for this model structure. It will turn out that these both exist in the case of $\mathrm{RMod}_O$ because it is a cofibrantly generated simplicial model category by Lemma \ref{lem:combinatorial}.

\begin{dfn}
    The homotopy coherent nerve $\mathcal{N}^{\mathrm{model}}(C)$ of a simplicial model category $C$ is the quasicategory $\mathcal{N}(C^{\mathrm{bi}})$.
\end{dfn}

\begin{dfn}
If $F: C \rightarrow V$ is a continuous functor of topologically enriched categories and $\mathrm{Sing}(V)$ is equipped with the structure of a simplicial model category with an enriched bifibrant replacement $\mathcal{B}$, then
\[\tilde{\mathcal{N}}(F): \mathcal{N}(\mathrm{Sing}(C)) \rightarrow \mathcal{N}^\mathrm{model}(\mathrm{Sing}(V))\]
is the composite $\mathcal{N}(\mathcal{B}\circ \mathrm{Sing}(F))$.

\end{dfn}
From now on we assume our simplicial model categories have a fixed enriched bifibrant replacement.
\begin{dfn}
    A Weiss cover $\mathcal{U}$ of $M \in \mathscr{M}\mathrm{fld}_n^*$ is an open cover of $M$ with the property that any finite subset $X$ of $M$ is contained in some $U \subset \mathcal{U}$.
\end{dfn}

A Weiss cover $\mathcal{U}$ can be interpreted as functor from the poset associated to $\mathcal{U}$ taking values in strictly framed embeddings: \[\mathrm{P}(\mathcal{U}) \rightarrow \mathscr{M}\mathrm{fld}_n^{\mathrm{fr}}.\] 

\begin{dfn}\label{dfn:cosheaf}
A topological Weiss cosheaf on $\mathscr{M}\mathrm{fld}_n^{\mathrm{fr}}$ valued in a topological model category $V$ is a continuous functor $F:\mathscr{M}\mathrm{fld}_n^{\mathrm{fr}} \rightarrow V$ which is a homotopy cosheaf with respect to Weiss covers, meaning for every Weiss cover $\mathcal{U}$ of $M \in \mathscr{M}\mathrm{fld}_n^{\mathrm{fr}}$:

\[\mathrm{hocolim}(P(\mathcal{U})  \rightarrow \mathscr{M}\mathrm{fld}_n^{\mathrm{fr}} \xrightarrow{F} V)\xrightarrow{\simeq} F(M)\]

\end{dfn}

The Weiss cosheaf condition allows one to study cosheaves on categories of manifolds by restricting attention to the subcategory of manifolds diffeomorphic to a disjoint union of disks. It turns out in the case of $\mathscr{M}\mathrm{fld}^\mathrm{fr}_n$, it is possible to restrict to a poset category.

\begin{dfn}
    For a smooth $n$-manifold $M$ the poset $\mathrm{Disk}(M)$ has objects the open subsets of $M$ diffeomorphic to $\bigsqcup_{i \in I} \mathbb{R}^n$, where $I$ is a finite set, and morphisms given by inclusion.
\end{dfn}

We now classify Weiss cosheaves on the category of framed manifolds in terms of their behavior on $\mathbb{R}^n$. A monoidal version of this statement appears in \cite[Theorem 5.4.5.9]{lurieHA}, and all the ideas from our proof are found in \cite{ayala_francis_2015,miller_ayala_francis_2020}. Recall that a quasicategory is presentable \cite[Definition 5.5.0.18]{lurie_2009} if it admits colimits and is accessible, meaning it is generated under ``small'' filtered colimits by a ``small'' category of ``small'' objects. In practice, most naturally occurring quasicategories are presentable.

\begin{prp}\label{prp:cosheaf}
    Suppose  $F,G$ are topological Weiss cosheaves on $\mathscr{M}\mathrm{fld}^{\mathrm{fr}}_n$ with values in a $V$ such that $\mathcal{N}^\mathrm{model}(V)$ is presentable, then there is an equivalence of functors $\tilde {\mathcal{N}} (F) \simeq \tilde {\mathcal{N}}(G)$, if and only if there is an equivalence of functors \[\tilde{\mathcal{N}}(F)|_{\mathcal{N}(\mathrm{Disk}(\mathbb{R}^n))} \simeq \tilde{\mathcal{N}}(G)|_{\mathcal{N}(\mathrm{Disk}(\mathbb{R}^n))}. \]
\end{prp}

\begin{proof}
By \cite[Theorem 4.2.4.1]{lurie_2009} homotopy colimits in $V$ and colimits in $\mathcal{N}^{\mathrm{model}}(V)$ agree. So following \cite[Proof of Proposition 2.22]{miller_ayala_francis_2020} the Weiss condition allows us to compute the value $\tilde{\mathcal{N}}(F)(M)$ as 
\[\mathrm{colim}(\mathcal{N}(\mathrm{Disk}(M)) \rightarrow \mathscr{M}\mathrm{fld}^{*}_n \xrightarrow{F} \mathcal{N}^{\mathrm{model}}(V)).\]

%(co)finality of disk(M) in the poset of the open cover it generates

If $W$ denotes the subcategories of isotopy equivalences in any of these quasicategories, then the above composite factors through the localization $\mathcal{N}(\mathrm{Disk}(M))[W^{-1}]$ \cite[Definition 1.3.4.1]{lurieHA} since the continuity of $F$ ensures that isotopy equivalences are inverted. We have the following commutative diagram where all maps are induced by inclusion or forgetting:
\begin{center}
% https://tikzcd.yichuanshen.de/#N4Igdg9gJgpgziAXAbVABwnAlgFyxMJZABgBoBGAXVJADcBDAGwFcYkQAdDgW3pwAsATt2AARLHADWAXwB6AKgD6wMAHoAstOQB1WcAC05aZRDTS6TLnyEUZYtTpNW7LrwEBjJsABy0gBSufEIi4lL+6gCUETp6hsam5iAY2HgERABMpPY0DCxsiJw8QZ6MPnKBAsLA3NAwjP4AahEJFinWROQUDrnOBRX8cO6CYtL9VRIyevKkYyIAZsyM9dKKhGatVmkomVQ5TvmFbsFiE-79AEbnwABKcmDRugZGJsY0MFAA5vBEoHOCENwkJ0QDgIEhiOsQH8AeCaKCkABmSHQwGIAAscLBiARezyLg42G4MAAjiAaIx6Oc6gAFSypGwgQRYD78HAtKH-VEYkFYzKOPF9Dh4RiwYD9EplfwAMWa5MpNLp7QKTJZbORnMRmKQ6XVMMQwPhiD5PQOXGFovFXl80uaxmkQA
\begin{tikzcd}
{\mathcal{N}(\mathrm{Disk}(M))[W^{-1}]} \arrow[d] \arrow[rr, "\tilde{\mathcal{N}}(F)"] &                                                          & \mathcal{N}^\mathrm{model}(V)                                                                   \\
{\mathrm{Disk}^*_{n/M}[W^{-1}]} \arrow[r]                                              & {\mathscr{D}\mathrm{isk}^{*,\mathrm{lluf}}_n} \arrow[ru] & {\mathcal{N}(\mathrm{Disk}(\mathbb{R}^n))[W^{-1}]} \arrow[l, "\simeq"'] \arrow[u, "\tilde{\mathcal{N}}(F)"']
\end{tikzcd}
\end{center}
The category $\mathscr{D}\mathrm{isk}^{*,\mathrm{lluf}}_n$, ``lluf'' being the reverse of ``full'', is defined as the closure of $\mathscr{D}\mathrm{isk}^*_n$ under equivalence, or in other words, the subcategory of framed manifolds diffeomorphic to a disjoint union of disks. Since the path following the bottom row depends only on $\tilde{\mathcal{N}}(F)$, it suffices to demonstrate the equivalence claimed in the diagram. This is \cite[Proposition 2.19]{ayala_francis_2015}.

\end{proof}

In practice, it is much easier to define functors on the subcategory of $\mathscr{M}\mathrm{fld}_n^{\mathrm{fr}}$ of strictly framed embeddings. For instance, the right modules $\mathcal{F}_M$ are functorial on this subcategory, but not the category $\mathscr{M}\mathrm{fld}_n^{\mathrm{fr}}$.

\begin{dfn}
    The category $\mathrm{Mfld}^\mathrm{strict}_n$ is the discrete category with objects framed $n$-manifolds and morphisms given by open embeddings which preserve the framing.
\end{dfn}

The ``downside'' of this category is that there are far fewer morphisms. In particular, if $B^n$ denotes the open unit ball with its standard framing and $|I| \geq 2$, then there are no strictly framed embeddings $\bigsqcup_I B^n \rightarrow B^n$. This is because a framed embedding is automatically an isometric embedding.

\begin{dfn}
    If $C \leq D$ is a faithful map of topological categories, then a homotopy extension of a continuous functor $F: C \rightarrow V$ is a continuous functor $\bar{F}: D \rightarrow V$ such that $\bar{F}|_C$ is connected by a zigzag of natural weak equivalences to $F$.
\end{dfn}

Recall that a functor $F:C_1 \rightarrow C_2$ between model categories ``creates homotopy colimits'' if a diagram $D$ in $C_1$ is a homotopy colimit diagram, if and only if $F(D)$ is a homotopy colimit diagram in $C_2$. Similarly, $F$ ``creates weak equivalences'' if a morphism $f$ in $C_1$ is a weak equivalence, if and only if $F(f)$ is a weak equivalence in $C_2$.

\begin{lem}\label{lem:liftingweiss}
    Let $T:V \rightarrow U$ be a continuous functor of topological model category such that $\mathcal{N}^\mathrm{model}(V),\mathcal{N}^\mathrm{model}(U)$ are presentable. Suppose $T$ creates homotopy colimits and weak equivalences. If $F,G: \mathrm{Mfld}^{\mathrm{strict}}_n \rightarrow V$ are functors such that $T \circ F,T \circ G$ admit homotopy extensions to $\mathscr{M}\mathrm{fld}_n^{\mathrm{fr}}$  which are Weiss cosheaves, then there is an equivalence \[\tilde {\mathcal{N}} (F) \simeq \tilde {\mathcal{N}}(G),\] if and only if, there is an equivalence of functors \[\tilde{\mathcal{N}}(F)|_{\mathcal{N}(\mathrm{Disk}(\mathbb{R}^n))} \simeq \tilde{\mathcal{N}}(G)|_{\mathcal{N}(\mathrm{Disk}(\mathbb{R}^n))}. \]

\end{lem}

\begin{proof}
    Since $T \circ F,T \circ G$ extend up to homotopy to $\mathscr{M}\mathrm{fld}_n^{\mathrm{fr}}$, we conclude that both $\mathcal{N}^\mathrm{model}(F),\mathcal{N}^\mathrm{model}(G)$ invert isotopy equivalences and so factor through $\mathrm{Mfld}^{\mathrm{strict}}_n[W^{-1}]$. In particular, their restrictions to $\mathcal{N}(\mathrm{Disk}(\mathbb{R}^n))$ factor through $\mathcal{N}(\mathrm{Disk}(\mathbb{R}^n))[W^{-1}]$ and so give rise to presheaves $\hat{F}\simeq \hat{G}$ on $\mathcal{D}\mathrm{isk}_n^*$ via the equivalence $\mathcal{D}\mathrm{isk}_n^* \simeq \mathcal{N}(\mathrm{Disk}(\mathbb{R}^n))[W^{-1}]$.

    We may left Kan extend $\hat{F},\hat{G}$ along the inclusion $i:\mathcal{D}\mathrm{isk}_n^* \rightarrow \mathscr{M}\mathrm{fld}_n^*$ to get comparisons
    \[F \leftarrow (i_!(\hat{F}))|_{\mathcal{N}(\mathrm{Mfld}^{\mathrm{strict}}_n)} \simeq (i_!(\hat{G}))|_{\mathcal{N}(\mathrm{Mfld}^{\mathrm{strict}}_n)} \rightarrow G.\]

Since $T$ preserves colimits and creates weak equivalences, $\tilde{\mathcal{N}}(T)$ will commute with left Kan extensions. Applying Proposition \ref{prp:cosheaf} to the coherent nerves of the homotopy extensions of $T \circ F,T \circ G$ will imply that the outer maps are also equivalences.
\end{proof}

We now supply a simplicial model structure to $\mathrm{RMod}_O$ and verify $\mathrm{Sing}(\mathrm{RMod}_O)$ is a simplicial model category with enriched bifibrant replacement. Following observations of Arone-Ching \cite[Appendix A]{arone_ching_2011}, we may model the category of right modules over the spectral operad $O$ as enriched presheaves on the $(\mathrm{Sp},\wedge)$-enriched category $\mathrm{Operator}(O)$ associated to $O$.\footnote{Recall the enriched category $\mathrm{Operator}(O)$ associated to the operad $O$ has objects given by finite sets and $\mathrm{Hom}(I,*)=O(I)$, with the rest of the morphisms obtained in a combinatorial manner. It is also called the PROP associated to $O$.}

%THIS IS NOT WHAT COMBINATORIAL MEANS!; dont actually need it here, but need it later. either replace with simplicial sets or give an additional argument; should be enough to have a simplicial equiv then work of may guillou shows the presheaf cat equiv

%I think it suffices that it is quillen equiv to a combinatorial simplicial model category such that it preserves all weak equiv. this allows one to make the hocolim arg and the rectification argument

 \begin{lem} \label{lem:combinatorial}
     The category $\mathrm{RMod}_O$ for $O \in \mathrm{Operad}(\mathrm{Sp},\wedge)$ admits a simplicial model structure which is cofibrantly generated and has weak equivalences and homotopy colimits computed objectwise. The simplicial mapping objects are $\mathrm{Sing}(\mathrm{Map}(-,-))$ where $\mathrm{Map}(-,-)$ denotes the $\mathrm{Top}$ enrichment inherited from the $\mathrm{Top}$ enrichment of orthogonal spectra. As a consequence, the category admits an enriched bifibrant replacement functor. The quasicategory $\mathcal{N}^\mathrm{model}(\mathrm{RMod}_O)$ is presentable.
 \end{lem}

 \begin{proof}
      The existence of a cofibrantly generated model structure is a consequence of \cite[Proposition 14.1.A]{fresse_2009} since the category of orthogonal spectra with the positive model structure is cofibrantly generated \cite[Theorem 14.1]{mandell_may_schwede_shipley_2001}. The simplicial enrichment follows formally as it does in the case of S-modules \cite[Proposition A.1]{arone_ching_2011} where the characterization of homotopy colimits can also be found. From this, one can deduce the existence of an enriched bifibrant replacement functor \cite[Theorem 13.5.2]{riehl_2014}. The nerve is presentable since cofibrant generation implies a Quillen equivalence with a combinatorial model category by \cite{rosicky}.\footnote{This assumes a large cardinal axiom called ``Vopënka’s principle'', likely one can avoid this by comparing this category of right modules to a category of right modules based in a combinatorial model category of spectra.} The simplicial localizations of combinatorial model categories are presentable \cite{pavlov2022combinatorial} 
      % it is accessible \cite{moser_2019} and admits colimits calculated objectwise \cite[Proposition A.1]{arone_ching_2011}.
 \end{proof}

% \begin{proof}
%      A presentable $\infty$-category is characterized as being equivalent to the nerve of a simplicial, combinatorial model category \cite[Proposition A.3.7.6]{lurie_2009}. Modulo local presentability of the underlying discrete category, this is contained in the above proposition. Local presentability follows from the fact that colimits are computed pointwise, and that the category is accessible \cite[Theorem 4.4]{moser_2019}.   
%  \end{proof}

Note that if $O$ is the trivial operad, this yields a model structure on $\Sigma \mathrm{Seq}(\mathrm{Sp})$ with the stated properties.

% \begin{lem} \label{lem:excise}
%     A continuous functor \[F:\mathscr{M}\mathrm{fld}_n^{\mathrm{fr}} \rightarrow \mathrm{RMod}_{O}\] is a Weiss cosheaf, if and only if the composition with the forgetful functor \[u:\mathrm{RMod}_{O} \rightarrow \Sigma\mathrm{Seq}(\mathrm{Sp})\]
%     is a Weiss cosheaf.
% \end{lem}

\begin{prp}\label{prp:firstexcise}
   The functor \[\Sigma^\infty_+ \mathcal{F}_{(-)}:\mathrm{Mfld}^{\mathrm{strict}}_n \rightarrow \Sigma\mathrm{Seq}(\mathrm{Sp})\] has a homotopy extension to a topological Weiss cosheaf on $\mathscr{M}\mathrm{fld}^{\mathrm{fr}}_n$ given by
   \[M \mapsto \Sigma^\infty_+ F(M,-).\]
\end{prp}

\begin{proof}
    First, observe that the configuration spaces are functorial with respect to framed embeddings by forgetting all framing information and using the functoriality of configuration spaces with respect to embeddings. Configuration spaces were observed to be excisive with respect to Day convolution of symmetric sequences \cite[Lemma 2.5]{campos_ricardo_idrissi}. By \cite[Proposition 3.14]{miller_ayala_francis_2020}, this implies it satisfies the Weiss cosheaf condition. It is clearly a homotopy extension since configuration spaces include as the interior of a manifold with boundary into the Fulton-MacPherson compactifications.
\end{proof}

\begin{prp}\label{prp:secondexcise}
 
    The functor \[s_{(n,n)}K(\Sigma^\infty \mathcal{F}_{(-)^+}):\mathscr{M}\mathrm{fld}^{\mathrm{strict}}_n \rightarrow \Sigma\mathrm{Seq}(\mathrm{Sp})\] has a homotopy extension to a topological Weiss cosheaf on $\mathscr{M}\mathrm{fld}^{\mathrm{fr}}_n$ given by
   \[M \mapsto (\Sigma^\infty F(U,-)^+)^\vee.\]
\end{prp}

\begin{proof}

Observed in \cite[Section 7]{malin2}, there are equivalences of symmetric sequences natural with respect to inclusion 
    \[M \mapsto s_{(n,n)}\Sigma^\infty K(\Sigma^\infty \mathcal{F}_{M^+})\]
    \[\uparrow{\simeq}\]
    \[M \mapsto s_{(n,n)}(\Sigma^\infty F(M,-)^+)^\vee.\]
    given by the dual of collapsing the subspace of the $W$-construction consisting of trees with an internal edge. The one point compactifications of configuration spaces are contravariantly functorial with respect to framed embeddings since $F(M,-)$ is functorial (by forgetting any framing information) and an open embedding of manifolds determines an open embedding of configuration spaces.
%     This equivalence is implemented by the collapse of the trees with more than one nonleaf vertex. The equivalence arises because the homotopy type of $B(\mathcal{F}_{M^+})$ can be computed as $\mathcal{F}_{\bar{M}}/\partial\mathcal{F}_{\bar M}$, where $\bar{M}$ is a compact manifold with interior $M$ \cite[Proposition 2.5]{FTW}, and this quotient is homeorphic to $F(M,-)^+$. Observe that when we compose the equivalences
%    \[B(W(E_{M^+}))\xrightarrow{\simeq} B(\mathcal{F}_{M^+} )\xrightarrow{\simeq}  F(M,-)^+,\]
%  the result is an equivalence which is natural with respect to all framed embeddings \footnote{This phenomenon occurs because arbitrary framed embeddings $N \rightarrow M$ do not yield module maps $\mathcal{F}_{M^+} \rightarrow \mathcal{F}_{N^+}$, but they do yield symmetric sequence maps.}. Using this observation, we have a zigzag of equivalences natural with respect to all framed embeddings:

% \[M \rightarrow s_{(n,n)}K(\Sigma^\infty E_{M^+})\]
% \[\downarrow{\simeq}\]
% \[M \rightarrow s_{(n,n)}K(\Sigma^\infty W(E_{M^+}))\]
% \[\uparrow{\simeq}\]
% \[M \rightarrow s_{(n,n)}(\Sigma^\infty F(M,-)^+)^\vee.\]
To demonstrate that this extension is a Weiss cosheaf, observe there is pairing equivariant with respect to symmetric group actions and natural with respect to framed embeddings:
    \[\Sigma^\infty_+ E_M \wedge F(M,-)^+ \rightarrow S^{n|-|} \simeq S_{(n,n)}\]
    \[(f,(y_i)) \mapsto f^{-1}(y_i)\]
 which is a duality pairing by comparison to the duality pairing of \cite[Theorem 2.3]{malin}. The adjoint is thus an equivalence and natural with respect to framed embeddings:
    \[M \mapsto s_{(n,n)} (\Sigma^\infty F(M,-)^+)^\vee\]
    \[\uparrow{\simeq}\]
    \[M \mapsto \Sigma^\infty_+ E_M.\]
    
    This last functor is equivalent to $M \mapsto \Sigma^\infty_+ F(M,-)$ by passing to the origin of each disk. This functor was just shown to be a Weiss cosheaf, and the property of being a Weiss cosheaf is invariant under zigzags of natural equivalences.

\end{proof}

% \begin{remark}
%     One should be slightly careful when using that the functor \[\mathscr{M}\mathrm{fld}^{\mathrm{fr}}_n:M \rightarrow \Sigma^\infty E_{M^+}\]
%     is continuous. The topology on the mapping spaces of $\mathscr{M}\mathrm{fld}^{\mathrm{fr}}_n$ is derived from the $C^\infty$-compact-open topology while the topology on $E_{M^+}$ is derived from the compact-open topology. Since the $C^\infty$-compact-open topology is finer than the compact-open topology, there is no issue.
% \end{remark}
Let $u:\mathrm{RMod}_O \rightarrow \Sigma \mathrm{Seq}(\mathrm{Sp})$ denote the forgetful functor. Note that it creates weak equivalences and homotopy colimits.
\begin{dfn}\label{dfn:natural}
    Given a zigzag map of operads $(O_1,f_1,O_2,f_2,\dots, O_{i-1},f_{k-1},O_k)$ and a sequence of topologically enriched functors $F_i:C \rightarrow \mathrm{RMod}_{O_i}$, a zigzag natural transformation $(F_1,\alpha_1,F_2,\alpha_2,\dots, F_{k-1},\alpha_{k-1},F_k)$ of enriched functors is a zigzag of natural transformations $\alpha_i:u \circ F_i \leftrightarrow u\circ F_{i+1}$ which objectwise determines a zigzag map of right modules. 
    
    A zigzag natural transformation is a zigzag natural equivalence if all $\alpha_i$ are objectwise zigzag equivalences, and it is a zigzag of natural equivalences if all $\alpha_i$ are objectwise equivalences.
\end{dfn}

In order to most easily apply the theory of Weiss cosheaves, we need a process which converts zigzag equivalences of right modules over different operads into zigzag equivalences over a single operad. 

\begin{dfn}\label{dfn:pullback}
Suppose we have a zigzag map $f=(O_1,f_1,\dots,O_k)$ of operads in $(\mathrm{Sp},\wedge)$  such that all $O_i$ are levelwise cofibrant as spectra with respect to the positive model structure on orthogonal spectra. Then for a right module $R$ over $O_k$, we define the restriction $\mathrm{res}_f(R)$ by setting $i=k$ and iterating the process:

\begin{enumerate}
  \item If $f_{i-1}$ is in the direction $O_{i-1} \rightarrow O_{i}$, replace $R$ with $R':=\mathrm{res}_{f_{i-1}}(R)$, otherwise:
  \item It must be a weak equivalence in the direction $O_{i} \xrightarrow{\simeq} O_{i-1}$, and we replace $R$ with $R'$ defined as the derived induction\footnote{Using the cofibrancy assumption on $O$, it can be computed by the bar construction $B(\mathcal{D}(R),O_{i},O_{i-1})$ where $\mathcal{D}$ denotes a fixed enriched cofibrant replacement functor. \cite[Proposition 8.5]{arone_ching_2011}.} of $R$ along $O_{i} \xrightarrow{\simeq} O_{i-1}$.
  \item Repeat this process with $R'$ and the zigzag equivalence of operads truncated at $O_{i-1}$.
\end{enumerate}

The process terminates with a right module over $O_1$ and this is defined as $\mathrm{res}_{f}(R)$.
\end{dfn}

By construction there is a zigzag map of right modules from $\mathrm{res}_f(R)$ to $R$ which is a zigzag equivalence if $f$ is a zigzag equivalence of operads.

By \cite{kro_2007}, a model structure on operads in orthogonal spectra exists and cofibrant operads are levelwise cofibrant. Let us fix a cofibrant replacement functor \[C:\mathrm{Operad}(\mathrm{Sp},\wedge) \rightarrow \mathrm{Operad}(\mathrm{Sp},\wedge).\]

\begin{dfn}
   The functor $C:\mathrm{RMod}_O \rightarrow \mathrm{RMod}_{C(O)}$ is restriction along $C(O)\rightarrow O$.
\end{dfn}

Thus, for any zigzag equivalence of operads $f=(O_1,f_1,\dots,O_k)$ and right modules $f'=(R_1,f'_1,\dots,R_k)$ there is an equivalent zigzag equivalence of levelwise cofibrant operads $C(f)=(C(O_1),C(f_1),\dots,C(O_k))$ and a compatible zigzag of right modules $C(f')=(C(R_1),C(f'_1),\dots,C(R_k))$. We emphasize: $(C(R_1),C(f'_1),\dots,C(R_k))$ has the same underlying symmetric sequence as $R_i$, it is only considered as a right module over a different operad.

Observe that if we have a zigzag natural equivalence of functors \[( D\rightarrow \mathrm{RMod}_{O_i},\alpha_i)\] compatible with a zigzag equivalence of operads $(O_1,f_1,\dots,O_k)$, we may use the above construction to pull back to a zigzag natural equivalence of functors taking values in $\mathrm{RMod}_{C(O_1)}$. Recall that $u:\mathrm{RMod}_O \rightarrow \Sigma \mathrm{Seq}(\mathrm{Sp})$ is the forgetful functor.

\begin{lem} \label{lem:pullback}
    Given a zigzag equivalence $f: (O_1,f_1,\dots,O_k)$ of operads in $(\mathrm{Sp},\wedge)$ and a continuous functor $F:\mathscr{M}\mathrm{fld}_n^{\mathrm{strict}} \rightarrow \mathrm{RMod}_{O_k}$, then $u \circ C \circ F=u \circ F$ admits a homotopy extension to $F':\mathscr{M}\mathrm{fld}_n^{\mathrm{fr}} \rightarrow \Sigma\mathrm{Seq}(\mathrm{Sp})$, if and only if 
    $u \circ \mathrm{res}_{C(f)} (C \circ F)$ also admits a homotopy extension to $F'$.

\end{lem}
\begin{proof}
By induction, it suffices to show the result holds for zigzags of length one. In other words, for restriction along a map of operads $N \rightarrow O$ and derived induction along an equivalence $O \xrightarrow{\simeq} P$. The first is automatic because restriction does not change the underlying homotopy type of the symmetric sequence. The second follows from the fact that derived induction of a levelwise cofibrant right module along a weak equivalence of cofibrant operads does not change the weak homotopy type of the underlying symmetric sequence \cite[Proposition 8.5]{arone_ching_2011}.
\end{proof}

\begin{lem}\label{lem:weissmodules}
Suppose that for $i=0,1$ $F_i: \mathrm{Mfld}^{\mathrm{strict}}_n \rightarrow \mathrm{RMod}_{O_i}$ are functors such that $u \circ F_i$ admit homotopy extensions to $\mathscr{M}\mathrm{fld}_n^{\mathrm{fr}}$  which are Weiss cosheaves. If there exists a zigzag natural equivalence $f$ from $F_1|_{\mathrm{Disk}(\mathbb{R}^n)}$  to $F_2|_{\mathrm{Disk}(\mathbb{R}^n)}$, then for any framed manifold $M$ 
 there is a zigzag equivalence of right modules from $F_1(M)$ to $F_2(M)$.
\end{lem}

\begin{proof}This theorem is implied by the following stronger statement which also encodes homotopy coherent naturality with respect to strictly framed embeddings:

``The functors $\tilde{\mathcal{N}}(C \circ F_1)$ and $\tilde{\mathcal{N}}(\mathrm{res}_{C(f)}(C \circ F_2))$ are equivalent in the quasicategory \[\mathrm{Fun}(\mathcal{N}(\mathrm{Mfld}^{\mathrm{strict}}_n), \mathcal{N}^{\mathrm{model}}(\mathrm{RMod}_{C(O_1)})).''\]

% Since $\mathrm{Sing}(\mathscr{M}\mathrm{fld}^{\mathrm{fr}}_n)$ is a small simplicial category and, by Proposition \ref{lem:combinatorial}, $\mathrm{RMod}_O$ is a combinatorial simplicial model category, we can apply \cite[Proposition 4.2.4.4]{lurie_2009} which in this situation states there is an equivalence of quasicategories \[\mathcal{N}^{\mathrm{model}}(\mathrm{Fun}^{\mathrm{Kan}}(\mathrm{Sing}(\mathscr{M}\mathrm{fld}^{\mathrm{fr}}_n), \mathrm{Sing}(\mathrm{RMod}_{C(s_n K(\Sigma^\infty_+ E_n))})))\]\[ \simeq\]\[\mathrm{Fun}(\mathcal{N}(\mathrm{Sing}(\mathscr{M}\mathrm{fld}^{\mathrm{fr}}_n)), \mathcal{N}^{\mathrm{model}}(\mathrm{RMod}_{C(s_n K(\Sigma^\infty_+ E_n))}))\]

% Here, the category of enriched functors is being endowed with the projective model structure, which has its weak equivalences determined levelwise. Hence, 

% can be lifted to an equivalence of bifibrant enriched functors. The enriched functors this theorem produce are necessarily bifibrant replacements of $\mathrm{res}_{C(f')}(C \circ \Sigma^\infty_+ E_{(-)})$ and $C \circ s_{(n,n)} K(\Sigma^\infty E_{(-)^+}) $, and so connected to them by a zigzag of equivalences in
% \[\mathrm{Fun}^{\mathrm{Kan}}(\mathrm{Sing}(\mathscr{M}\mathrm{fld}^{\mathrm{fr}}_n), \mathrm{Sing}(\mathrm{RMod}_{C(s_n K(\Sigma^\infty_+ E_n))}))\]

To see this, first observe that by Definition \ref{dfn:pullback}, the functor $\mathrm{res}_{C(f)}(C \circ F_2) $ is related to $F_2$ by a zigzag natural equivalence (Definition \ref{dfn:natural}), and similarly for $F_1$ and $C \circ F_1$. Now recall that an equivalence in the quasicategory of functors is an objectwise equivalence \cite[Chapter 5, Theorem C]{joyal_2008}, and an equivalence of objects in $\mathcal{N}^\mathrm{model}(-)$ of a simplicial model category implies the existence of a weak equivalence in the simplicial model category of the bifibrant objects the vertices represent, and so we've shown the implication.\footnote{If in Lemma \ref{lem:combinatorial} we established that $\mathrm{Sing}(\mathrm{RMod}_O)$ was in fact simplicially equivalent to a combinatorial simplicial model category, we would get the stronger statement that these zigzags of objects are actually natural.}

 We now demonstrate the quasicategorical statement. Applying $C(-)$ to the zigzag natural equivalence in the statement of the theorem and restricting along the operad zigzag equivalence $C(f)$, yields a zigzag natural equivalence \[ C  \circ F_1 |_{\mathrm{Disk}(\mathbb{R}^n)}\quad \mathrm{to} \quad  \mathrm{res}_{C(f)}(C \circ F_2)|_{\mathrm{Disk}(\mathbb{R}^n)}  \]

Applying $\tilde{\mathcal{N}}$, we now have a zigzag of transformations between functors of quasicategories for which the backwards maps are equivalences objectwise. As before, the backwards natural transformations are in fact equivalences of functors, and thus can be inverted. Inverting the backwards arrows and composing yields a natural transformation
 \[\tilde{\mathcal{N}}(C \circ F_1 )|_{\mathcal{N}(\mathrm{Disk}(\mathbb{R}^n))}  \Rightarrow  \tilde{\mathcal{N}}(\mathrm{res}_{C(f)}(C \circ F_2))|_{\mathcal{N}(\mathrm{Disk}(\mathbb{R}^n))} . \]

By the definition of a zigzag equivalence (Definition \ref{dfn:zigzag}), this is an objectwise equivalence, and thus an equivalence of functors. As remarked earlier, the forgetful functor $u: \mathrm{RMod}_O \rightarrow \Sigma\mathrm{Seq}(\mathrm{Sp})$ creates homotopy colimits and weak equivalences, so since Lemma \ref{lem:pullback} shows these functors still have homotopy extensions, we can apply Lemma \ref{lem:liftingweiss} to conclude that the functors $\tilde{\mathcal{N}}(C \circ F_1)$ and $\tilde{\mathcal{N}}(\mathrm{res}_{C(f)}(C \circ F_2))$ are equivalent in the quasicategory \[\mathrm{Fun}(\mathcal{N}(\mathrm{Mfld}^{\mathrm{strict}}_n), \mathcal{N}^{\mathrm{model}}(\mathrm{RMod}_{C(O_1)})).\]

\end{proof}

\section{Self duality of $E_M$, Poincaré-Koszul duality, and embedding calculus}\label{sec:results}

In this section, we combine the results of the previous two sections to prove that the right modules $\mathcal{F}_M$ have compactly supported Koszul self duality, and that this is natural with respect to framed embeddings. We will discuss some applications including: a resolution of Ching's conjecture, a lift of the Pontryagin-Thom collapse map to stable, framed embedding calculus, and Poincaré-Koszul duality for left $\Sigma^\infty_+ \mathcal{F}_n$-modules.

\begin{thm}[Koszul self duality of $\mathcal{F}_M$] \label{thm:selfduality}
There is a zigzag of equivalences of operads
\[\Sigma^\infty_+ \mathcal{F}_n \simeq  \dots \simeq s_n K(\Sigma^\infty_+ \mathcal{F}_n)\]
and a compatible zigzag of equivalences of right modules
\[\Sigma^\infty_+ \mathcal{F}_M \simeq \dots \simeq s_{(n,n)} K(\Sigma^\infty \mathcal{F}_{M^+})\]

\end{thm}

\begin{proof}
 The restrictions of the functors \[ \Sigma^\infty_+ \mathcal{F}_{(-)}:\mathrm{Mfld}^{\mathrm{strict}}_n \rightarrow \mathrm{RMod}_{\Sigma^\infty_+ \mathcal{F}_n}\]
\[M \mapsto \Sigma^\infty_+ \mathcal{F}_M,\]
\[ s_{(n,n)}K(\Sigma^\infty \mathcal{F}_{(-)^+}):\mathrm{Mfld}^{\mathrm{strict}}_n \rightarrow \mathrm{RMod}_{s_{n}K(\Sigma^\infty_+ \mathcal{F}_n)}\]
\[M \mapsto s_{(n,n)}K(\Sigma^\infty \mathcal{F}_{M^+})\]
to $\mathrm{Disk}(\mathbb{R}^n)$ are connected by a zigzag natural equivalence by Lemma \ref{lem:codimension0}. After composition with $u$, the two functors admits topological Weiss cosheaf homotopy extensions by Lemma \ref{prp:firstexcise} and Lemma $\ref{prp:secondexcise}$. Thus we can apply, Lemma \ref{lem:weissmodules} and conclude the result.
    
\end{proof}

The self duality of $\mathcal{F}_M$ has a multitude of consequences which we begin to investigate. Recall that we can construct a zigzag map of operads $\mathrm{lie} \rightarrow s_{-n}\Sigma^\infty_+ \mathcal{F}_n$ by taking the Koszul dual of the standard map $\Sigma^\infty_+\mathcal{F}_n \rightarrow \mathrm{com}$ and appealing to the Koszul self duality of $\mathcal{F}_n$. We denote the pullback along this zigzag by $\mathrm{res}_{\mathrm{lie}}$. Recall from Section \ref{sec:conjecture}, the derivatives $\partial_* F$ of a functor $F:\mathrm{Top}_* \rightarrow \mathrm{Sp}$ form a right module of the $\mathrm{lie}$ operad.

\begin{cor}[Ching's conjecture]\label{cor:chingconjecture}
    If $M$ is a framed $n$-manifold, there is a zigzag of equivalences of right $ \mathrm{lie}$-modules \[\mathrm{res}_{\mathrm{lie}}(s_{(-n,-n)}\Sigma^\infty_+ \mathcal{F}_M) \simeq \dots \simeq \partial_*(\Sigma^\infty \mathrm{Map}_*(M^+,-)).\]
\end{cor}

\begin{proof}
    Recall from Section \ref{sec:conjecture} that $\partial_*(\Sigma^\infty \mathrm{Map}_*(M^+,-))= B((M^+)^\wedge,\mathrm{com},1)^\vee$ \cite[Example 17.28]{arone_ching_2011}. As a consequence of \cite[Proposition 2.5]{FTW} there is an equivalence of right $B(\mathrm{com})$-comodules 
    \[B(\mathcal{F}_{M^+},(\mathcal{F}_n)_+,1)\xrightarrow{\simeq} B((M^+)^\wedge,\mathrm{com},1)\] derived from the fact\footnote{A similar calculation appears in \cite[Proposition 3.17]{aroneching2014manifolds}.} $\mathcal{F}_{M^+}/\mathrm{decom}(\mathcal{F}_{M^+})\cong (M^+)^\wedge / \Delta^{\mathrm{fat}}$ . Dualizing this equivalence and applying the compactly supported Koszul self duality of $\mathcal{F}_{M}$ yields the result.
\end{proof}

Recall that Boavida de Brito--Weiss and Turchin showed that embedding calculus is computed as derived mapping spaces of right modules over a framed variant of the little disks operad \cite{Turchin_2013,brito_weiss_2013}. Using Koszul self duality, we can replicate the collapse map associated to a codimension $0$ embedding $M \rightarrow N$ for an arbitrary right module map $\Sigma^\infty_+ E_M \rightarrow \Sigma^\infty_+ E_N$.

%need slight cofibrancy hypothesis on operad to pull derived mapping SPACES across equivalences. currently don't know how to do it for spectra

\begin{thm}[Pontryagin-Thom collapse for stable, framed embedding calculus]\label{thm:pontry}
    For framed $n$-manifolds $M,N$ there is a map
    \[\mathrm{Map}^h_{\Sigma^\infty_+ \mathcal{F}_n}(\Sigma^\infty_+ \mathcal{F}_M, \Sigma^\infty_+ \mathcal{F}_N) \rightarrow \mathrm{Map}^h_{\Sigma^\infty_+ \mathcal{F}_n}(\Sigma^\infty \mathcal{F}_{N^+}, \Sigma^\infty \mathcal{F}_{M^+}).  \]
This map is an equivalence assuming Conjecture \ref{thm:chingKoszul} for $\Sigma^\infty_+ \mathcal{F}_n$.

\end{thm}
It is not difficult to see that in the case $\mathcal{F}_M \rightarrow \mathcal{F}_N$ is induced by a codimension 0 inclusion $i: M \rightarrow N$, the image under the above map is homotopic to \[\Sigma^\infty i^+: \Sigma^\infty \mathcal{F}_{N^+} \rightarrow \Sigma^\infty \mathcal{F}_{M^+}.\]
Before giving the construction, we recall a conjecture about Koszul duality for arbitrary operads.
%This might need some divided power hypothesis (in general) Proposition 2.7 of k-theory, manifolds,...!
\begin{conjecture}[Folklore]\label{thm:chingKoszul}
    For a levelwise $\Sigma$-finite, levelwise cofibrant operad $O$ and levelwise $\Sigma$-finite, levelwise cofibrant right module $R$ in $(\mathrm{Sp},\wedge)$, there is a natural zigzag of equivalences of operads and compatible natural zigzag of equivalences of right modules
    \[O \simeq \dots \simeq K(K(O))\]
    \[R \simeq \dots \simeq K(K(R)).\]
    
\end{conjecture}

 A version of this result was proven for $\Sigma^\infty_+ E_n$ in \cite[Proposition 2.7]{aroneching2014manifolds}, though the Koszul duality functors there are of a subtantially different form from the Koszul duality functors of Ching's thesis. A consequence of this conjecture would be that (when derived) $K(-)$ is an equivalence on the subcategory of levelwise finite right modules, and thus:
\[\operatorname{Map}^h_{O}(R,R') \simeq \operatorname{Map}^h_{K(O)}(K(R'),K(R)).\]

The stated zigzag of operad equivalence is one of the main results of \cite{ching_2012} where it was also proven $R \simeq K(K(R))$ on the level of symmetric sequences.

\begin{proof}[Proof of Theorem \ref{thm:pontry}]
    Koszul duality supplies us with a map
    \[\mathrm{Map}^h_{\Sigma^\infty_+ \mathcal{F}_n}(\Sigma^\infty_+ \mathcal{F}_M,\Sigma^\infty_+ \mathcal{F}_N) \rightarrow \mathrm{Map}^h_{K(\Sigma^\infty_+ \mathcal{F}_n)}( K(\Sigma^\infty_+ \mathcal{F}_N),K(\Sigma^\infty_+ \mathcal{F}_M))\]
     If Conjecture \ref{thm:chingKoszul} is true, this map is an equivalence. The compactly supported Koszul self duality of $\mathcal{F}_M$ implies that, up to zigzags of equivalence compatible with the self duality of $\mathcal{F}_n$, $s_{(n,n)}K(\Sigma^\infty_+ \mathcal{F}_M)\simeq \Sigma^\infty \mathcal{F}_{M^+}$ which allows us to replace the second mapping space with $\mathrm{Map}^h_{\Sigma^\infty_+ \mathcal{F}_n}(\Sigma^\infty \mathcal{F}_{N^+},\Sigma^\infty \mathcal{F}_{M^+})$ since $s_{(n,n)}$ is invertible up to homotopy. 
\end{proof}

The theory of Koszul duality developed by Ching has an extension the category of left modules (without a $0$ term). We refer to \cite{ching_2005} for a detailed account.\footnote{Be warned, left modules are not defined via partial composites.} The Koszul duality of operads, right modules, and left modules has an interesting interaction with operadic bar constructions. For any level-cofibrant operad $O$, level-cofibrant right module $R$, and level-cofibrant left module $L$ in $(\mathrm{Sp},\wedge)$ there is an equivalence \cite[Proposition 6.1]{ching_2012}:
\[B(R,O,L) \xrightarrow{\simeq} \Omega(B(R),B(O),B(L))\]
Using the suggestive notation of \cite{amabel_2022} this can be written as
\[\int_{R} L \xrightarrow{\simeq} \int^{B(R)}B(L)\]
in analogy with the Poincaré-Koszul duality arrow of \cite{PKayala_francis_2019}. For a general operad $O$, one cannot further this analogy, however, in the case of $\mathcal{F}_M$, Koszul self duality of $\mathcal{F}_n$ and $\mathcal{F}_M$ allows us to interpret the righthand calculation as taking place over a suspension of the right comodule $(\Sigma^\infty \mathcal{F}_{M^+})^\vee$ over a suspension of the cooperad $ (\Sigma^\infty_+ \mathcal{F}_n)^\vee$.\footnote{In general, the dual of an operad is not naturally a cooperad since taking Spanier--Whitehead duals only distributes over smash products up to weak equivalence rather than on the nose. Here $O^\vee$ should be interpreted as $B(K(O))$ \cite[Remark 4.6.]{ching_2012}.} This parallels the fact that Poincaré-Koszul duality \cite{PKayala_francis_2019} relates $E_n$-algebra computations over $M$ to $E_n$-coalgebra computations over $M^+$. Implicitly restricting $B(L)$ along the zigzag equivalence $B(\Sigma^\infty_+ \mathcal{F}_n) \simeq s_n (\Sigma^\infty_+ \mathcal{F}_n) ^\vee$, we have:

\begin{thm}[Poincaré-Koszul duality for left $\Sigma^\infty_+ \mathcal{F}_n$-modules]\label{thm:Poincarékoszul}
For a framed $n$-manifold $M$ and a left $\Sigma^\infty_+ \mathcal{F}_n$-module $L$ there is an equivalence
\[\int_{\Sigma^\infty_+ \mathcal{F}_M} L \xrightarrow{\simeq} \int^{s_{(n,n)}\Sigma^\infty \mathcal{F}_{M^+}^\vee} B(L).\]

\end{thm}

 \bibliographystyle{plain}
 \bibliography{main.bib}

\begin{thebibliography}{10}

\bibitem{amabel_2022}
Araminta Amabel.
\newblock Poincar{é}/{K}oszul duality for general operads.
\newblock {\em Homology, Homotopy and Applications}, 24(2):1–30, 2022.

\bibitem{arone_ching_2011}
Gregory Arone and Michael Ching.
\newblock {\em Operads and chain rules for the calculus of functors}.
\newblock Société Mathématique De France, 2011.

\bibitem{aroneching2014manifolds}
Gregory Arone and Michael Ching.
\newblock Manifolds, {K}-theory and the calculus of functors, 2014.
\newblock \url{https://arxiv.org/abs/1410.1809}.

\bibitem{arone_ching_2015}
Gregory Arone and Michael Ching.
\newblock A classification of {T}aylor towers of functors of spaces and spectra.
\newblock {\em Advances in Mathematics}, 272:471–552, 2015.

\bibitem{ayala_francis_2015}
David Ayala and John Francis.
\newblock Factorization homology of topological manifolds.
\newblock {\em Journal of Topology}, 8(4):1045–1084, 2015.

\bibitem{PKayala_francis_2019}
David Ayala and John Francis.
\newblock Poincar{é}/{K}oszul duality.
\newblock {\em Communications in Mathematical Physics}, 365(3):847–933, 2019.

\bibitem{ayala_francis_2019}
David Ayala and John Francis.
\newblock Zero-pointed manifolds.
\newblock {\em Journal of the Institute of Mathematics of Jussieu}, 20(3):785–858, 2019.

\bibitem{miller_ayala_francis_2020}
David Ayala and John Francis.
\newblock {\em A factorization homology primer}.
\newblock CRC Press, 2020.

\bibitem{behrens_2012}
Mark Behrens.
\newblock The {G}oodwillie tower and the {EHP} sequence.
\newblock {\em Memoirs of the American Mathematical Society}, 218(1026), 2012.

\bibitem{boardman_vogt_1973}
J.~M. Boardman and R.~M. Vogt.
\newblock Homotopy invariant algebraic structures on topological spaces.
\newblock {\em Lecture Notes in Mathematics}, 1973.

\bibitem{brito_weiss_2013}
Pedro Boavida~de Brito and Michael Weiss.
\newblock Manifold calculus and homotopy sheaves.
\newblock {\em Homology, Homotopy and Applications}, 15(2):361–383, 2013.

\bibitem{brantner}
Lukas Brantner.
\newblock {\em The Lubin-Tate Theory of Spectral Lie Algebras}.
\newblock PhD thesis, Harvard University, 2017.

\bibitem{campos_ricardo_idrissi}
Ricardo Campos, Najib Idrissi, and Thomas Willwacher.
\newblock Configuration spaces of surfaces, 2019.
\newblock \url{https://arxiv.org/abs/1911.12281}.

\bibitem{ching}
Michael Ching.
\newblock Calculus of functors and configuration spaces.
\newblock Conference on Pure and Applied Topology Isle of Skye, \url{https://mching.people.amherst.edu/Work/skye.pdf}.

\bibitem{ching_2005}
Michael Ching.
\newblock Bar constructions for topological operads and the {G}oodwillie derivatives of the identity.
\newblock {\em Geometry \& Topology}, 9(2):833–934, 2005.

\bibitem{ching_2012}
Michael Ching.
\newblock Bar-cobar duality for operads in stable homotopy theory.
\newblock {\em Journal of Topology}, 5(1):39–80, 2012.

\bibitem{ching_salvatore}
Michael Ching and Paolo Salvatore.
\newblock Koszul duality for topological {$E_n$} operads, 2020.
\newblock \url{https://arxiv.org/abs/2002.03878}.

\bibitem{cohen_lada_may_1976}
Frederick~R. Cohen, Thomas~J. Lada, and J.~Peter May.
\newblock {\em The homology of iterated loop spaces}.
\newblock Springer-Verlag, 1976.

\bibitem{fresse_2009}
Benoit Fresse.
\newblock {\em Modules over operads and functors}.
\newblock Springer Berlin, 2009.

\bibitem{FTW}
Benoit Fresse, Victor Turchin, and Thomas Willwacher.
\newblock On the rational homotopy type of embedding spaces of manifolds in $\mathbb{R}^n$, 2020.
\newblock \url{https://arxiv.org/abs/2008.08146}.

\bibitem{getzler_jones}
Ezra Getzler and J.~D.~S. Jones.
\newblock Operads, homotopy algebra and iterated integrals for double loop spaces, 1994.
\newblock \url{https://arxiv.org/abs/hep-th/9403055}.

\bibitem{goodwillie_klein_weiss_2001}
Thomas~G. Goodwillie, John~R. Klein, and Michael~S. Weiss.
\newblock Spaces of smooth embeddings, disjunction and surgery.
\newblock {\em Princeton University Press eBook Package 2014}, page 221–284, 2001.

\bibitem{joyal_2008}
Andre Joyal.
\newblock Advanced course on simplicial methods in higher categories, 2008.
\newblock \url{https://mat.uab.cat/~kock/crm/hocat/advanced-course/Quadern45-2.pdf}.

\bibitem{klein_2001}
John~R. Klein.
\newblock The dualizing spectrum of a topological group.
\newblock {\em Mathematische Annalen}, 319(3):421–456, 2001.

\bibitem{klein_2007}
John~R Klein.
\newblock The dualizing spectrum {II}.
\newblock {\em Algebraic \& Geometric Topology}, 7(1):109–133, 2007.

\bibitem{knudsen_2018}
Ben Knudsen.
\newblock Higher enveloping algebras.
\newblock {\em Geometry \& Topology}, 22(7):4013–4066, 2018.

\bibitem{konovalov2023algebraic}
Nikolay Konovalov.
\newblock Algebraic {G}oodwillie spectral sequence, 2023.
\newblock \url{ https://arxiv.org/abs/2303.06240}.

\bibitem{kro_2007}
Tore~August Kro.
\newblock Model structure on operads in orthogonal spectra.
\newblock {\em Homology, Homotopy and Applications}, 9(2):397–412, 2007.

\bibitem{lind_malkiewich_2018}
John Lind and Cary Malkiewich.
\newblock The {M}orita equivalence between parametrized spectra and module spectra.
\newblock {\em New Directions in Homotopy Theory}, page 45–66, 2018.

\bibitem{loday_vallette_2012}
Jean-Louis Loday and Bruno Vallette.
\newblock {\em Algebraic Operads}.
\newblock Springer, 2012.

\bibitem{lurie_2009}
Jacob Lurie.
\newblock {\em {Higher Topos Theory}}.
\newblock Annals of Mathematics Studies. Princeton University Press, Princeton, NJ, 2009.

\bibitem{lurieVerd}
Jacob Lurie.
\newblock Poincar{é} spaces and {S}pivak fibrations.
\newblock \url{https://www.math.ias.edu/~lurie/287xnotes/Lecture26.pdf}, 04 2011.

\bibitem{lurieHA}
Jacob Lurie.
\newblock {\em Higher Algebra}.
\newblock 2017.
\newblock \url{https://people.math.harvard.edu/~lurie/papers/HA.pdf}.

\bibitem{malin}
Connor Malin.
\newblock An elementary proof of the homotopy invariance of stabilized configuration spaces, 2022.
\newblock \url{https://arxiv.org/abs/2208.05947}.

\bibitem{malin2}
Connor Malin.
\newblock The stable embedding tower and operadic structures on configuration spaces, 2022.
\newblock \url{https://arxiv.org/abs/2211.12654}.

\bibitem{malkbasic}
Cary Malkiewich.
\newblock Parametrized spectra, a low-tech approach, 2019.
\newblock \url{https://arxiv.org/abs/1906.04773}.

\bibitem{mandell_may_schwede_shipley_2001}
M.~A. Mandell, J.~P. May, S.~Schwede, and B.~Shipley.
\newblock Model categories of diagram spectra.
\newblock {\em Proceedings of the London Mathematical Society}, 82(2):441–512, 2001.

\bibitem{markl_1999}
Martin Markl.
\newblock A compactification of the real configuration space as an operadic completion.
\newblock {\em Journal of Algebra}, 215(1):185–204, 1999.

\bibitem{may_sigurdsson_2006}
J.~May and J.~Sigurdsson.
\newblock Parametrized homotopy theory.
\newblock {\em Mathematical Surveys and Monographs}, 2006.

\bibitem{pavlov2022combinatorial}
Dmitri Pavlov.
\newblock Combinatorial model categories are equivalent to presentable quasicategories, 2022.

\bibitem{riehl_2014}
Emily Riehl.
\newblock {\em Categorical homotopy theory}.
\newblock Cambridge University Press, 2014.

\bibitem{rosicky}
J.~Rosicky.
\newblock Are all cofibrantly generated model categories combinatorial ?
\newblock {\em Cahiers de Topologie et G\'eom\'etrie Diff\'erentielle Cat\'egoriques}, 50(3), 2009.

\bibitem{salvatore_1998}
Paolo Salvatore.
\newblock {\em Configuration operads, minimal models and rational curves}.
\newblock PhD thesis, University of Oxford, 1998.

\bibitem{salvatore_2021}
Paolo Salvatore.
\newblock The {F}ulton–{M}ac{P}herson operad and the {W}-construction.
\newblock {\em Homology, Homotopy and Applications}, 23(2):1–8, 2021.

\bibitem{spivak_1967}
Michael Spivak.
\newblock Spaces satisfying {P}oincar{é} duality.
\newblock {\em Topology}, 6(1):77–101, 1967.

\bibitem{Turchin_2013}
Victor Turchin.
\newblock Context-free manifold calculus and the {Fulton–Macpherson operad}.
\newblock {\em Algebraic \& Geometric Topology}, 13(3):1243–1271, 2013.

\bibitem{Volpe}
Marco Volpe.
\newblock Verdier duality on conically smooth stratified spaces, 2022.
\newblock \url {https://arxiv.org/abs/2206.02728}.

\bibitem{wall_1967}
C.~T. Wall.
\newblock Poincar{é} complexes: I.
\newblock {\em The Annals of Mathematics}, 86(2):213, 1967.

\bibitem{weiss_1999}
Michael Weiss.
\newblock Embeddings from the point of view of immersion theory: Part {I}.
\newblock {\em Geometry \& Topology}, 3(1):67–101, 1999.

\end{thebibliography}

 \end{document}